\documentclass[11pt]{amsart}

\usepackage{enumitem}
\usepackage{xcolor,graphicx}
\usepackage{hyperref}
\usepackage{tikz}
\usepackage{amsmath, amsfonts, latexsym, array, amssymb, amscd }
\usetikzlibrary{shapes,arrows,positioning,decorations.pathreplacing}

\usepackage{fancyhdr,datetime2}
\newcommand{\Ar}
{\mathcal{L}_A}
\newcommand{\ph}{\varphi}

\newcommand{\eps}{\varepsilon}

\newcommand{\RR}{\mathbb{R}}

\newcommand{\PPP}{\mathcal{P}}
\newcommand{\SSS}{\mathcal{S}}
\newcommand{\BBB}{\mathcal{B}}
\newcommand{\UUU}{\mathcal{U}}
\newcommand{\GGG}{\mathcal{G}}

\newcommand{\RRR}{\mathcal{R}}

\newcommand{\supp}{\text{supp}}
\def\ae{\text{-a.e.}\ }

\newcommand{\NN}{\mathbb{N}}

\usepackage{ulem}

\numberwithin{equation}{section}

\newtheorem{theorem}{Theorem}[section]
\newtheorem{proposition}[theorem]{Proposition}
\newtheorem{lemma}[theorem]{Lemma}
\newtheorem{definition}[theorem]{Definition}
\newtheorem{remark}[theorem]{Remark}
\newtheorem{corollary}[theorem]{Corollary}

\newtheorem{thma}{Theorem}

\begin{document}

\title[Maximal measures for flows with nonuniform structure]{Maximal measures for flows with non-uniform structure}

\author{Qiao Liu}
\author{Tianyu Wang}
\author{Weisheng Wu}

\address{School of mathematics and computational science, Xiangtan university, Xiangtan, Hunan 411105, P.R. China}
\email{qiaoliu@xtu.edu.cn}

\address{School of Mathematical Sciences, Shanghai Jiaotong University, No 800 Dongchuan Road, Shanghai 200240, P.R.China}
\email{a356789xe@sjtu.edu.cn}

\address{School of Mathematical Sciences, Xiamen University, Xiamen, 361005, P.R. China}
\email{wuweisheng@xmu.edu.cn}

\begin{abstract}
In this paper, we study ergodic optimization of continuous functions for flows by concentrating on the entropy spectrum of their maximizing measures. Precisely, over a wide family of flows with non-uniformly hyperbolic structure, we obtain a picture describing coexistence of continuous functions whose maximizing measures have large and small entropy respectively in $C^0$-topology. Our proof relies on the orbit decomposition technique, originally introduced by Climenhaga and Thompson, for flows with weakened versions of expansiveness and specification property. In particular, our results extend \cite{STY} from non-Markov shift on symbolic spaces to a considerably broad class of continuous flows with nonuniform structure. To illustrate this, we apply our general results to both geodesic flows and frame flows over closed rank one manifolds of nonpositive curvature.
\end{abstract}



\maketitle

\section{Introduction}

Let $(X,F)$ denote a continuous flow $F=\{f_t \}_{t \in \RR}$ or a homeomorphism 
on a compact metric space $X$, and $\mathcal{M}_F(X)$ (resp. $\mathcal{M}_F^e(X)$) denote the space of $F$-invariant (resp. $F$-ergodic) Borel probability measure on $X$ equipped with the weak-* topology. Denoting by $C(X, \RR)$ the space of real-valued continuous functions on $X$, we define for each $\ph  \in C(X, \RR) $ its maximum functional 
\begin{equation}\label{max}
\Lambda_F(\ph):=\sup \Big\{\int \ph d\nu: \nu \in \mathcal{M}_F(X)\Big\} 
\end{equation}
and the set of all maximizing measure  of  $\ph$, 

$$\mathcal{M}_{\max}(\ph):=\Big\{\mu \in  \mathcal{M}_F(X) : \int \ph d\mu=\Lambda_F(\ph)\Big\}.$$

The study of properties for maximizing measures, known as ergodic optimization nowadays, originated in the 1990s and has attracted many research interests. Of particular interest is to establish properties of $\mathcal{M}_{\max}(\ph)$ when $\ph$ belongs to a generic\footnote{Throughout this paper, we say a property is \emph{generic} if it holds for a dense $G_{\delta}$ set.} subset of Banach space consisting of suitably regular continuous functions. One such specified property is that maximizing measures are supported on periodic orbits. The possibility of typical maximizing measures being periodic was suggested by early work on ergodic optimization \cite{Bo00, HO96a, HO96b} and later formulated as Typically Periodic Optimization conjecture. The regularity of continuous functions plays a crucial role on the statistical property of maximizing measures. In fact, typical properties of maximizing measures in the space of continuous functions are rather different from those of more regular ones like  H\"older or Lipschitz functions. 

Another significant strand of research in ergodic optimization is by interpreting maximizing measures as the zero temperature limits of equilibrium states. Recall that equilibrium state is an invariant probability measure maximizing the quantity $h_\mu(F)+\int \ph d\mu$. Replacing $\ph$ by $\alpha\ph$ for $\alpha \in \RR$, as $\alpha \rightarrow \infty$, we observe that the entropy term loses relative importance and homogeneity of maximizing measure implies that  for large $\alpha$, an equilibrium measure $\mu_{\alpha\ph}$ for $\alpha\ph$ is almost maximizing for $f$ (the thermodynamic interpretation of the parameter $\alpha$ is as an inverse temperature, so that letting $\alpha \rightarrow \infty$ is referred to as a zero temperature limit). Following this philosophy, a reasonable prerequisite in many settings would be the existence and uniqueness of equilibrium state. Such properties can be established by certain dynamical behaviors of both the system and the function (the latter is often called the potential in thermodynamic formalism), whose ideas trace back to the remarkable work of Bowen in 1970’s \cite{Bow74}. In this work, he showed that in the case of homeomorphism $(X, F, \ph)$, expansiveness and specification of $(X, F )$ plus a bounded distortion property of $\ph$ (also known as the Bowen property nowadays) would indicate existence and uniqueness of equilibrium state for $(X, F, \ph)$. Bowen's criteria was improved by Climenhaga and Thompson \cite{CT16} by relaxing the above three uniform conditions. 

Recently, Shinoda, Takahasi and Yamamoto \cite{STY} showed that the space of continuous functions over shift spaces with certain non-Markov structures contains two disjoint subsets: the first one is a generic set for which all maximizing measures have ``relatively small'' entropy; the second one is the set of functions having uncountably many, fully supported ergodic maximizing measures with ``relatively large'' entropy. Though possibly not being generic by itself, the second set is still rich in the sense of being dense among the complement of the first set.

In this paper, we study ergodic optimization of continuous functions from the perspective of entropy spectrum of generic maximizing measures, by extending the above symbolic results in \cite{STY} to a considerably wide family of flows. All systems we concern possess a variety of non-uniformly hyperbolic structures, with our major emphasis over two common non-uniform versions of expansiveness and specification property. With the development on the theory of non-uniform expansiveness and specification in the past two decades, people have come to realize that systems possessing these sorts of properties may display all different categories of non-uniform hyperbolicity; these include systems with dominated splittings that are partially hyperbolic or not, and other derived-from-Anosov systems (see \cite{CFT18,CFT19,CT21, Wang21, BCFT, PYY}). Therefore, along with results in \cite{STY}, our results demonstrate the rich structure of ergodic optimization beyond uniformly hyperbolic or Markov systems.

\subsection{General results}
Let $(X,F)$ be a continuous flow on a compact metric space $(X,d)$ which does not possess fixed points\footnote{Though the entire paper is in the context of flows, our results can be easily adapted to those homemorphisms as all our strategies for flows are applicable there, following relatively simpler arguments.}. We mainly concern the metric entropy of maximizing measures of continuous functions for flows with non-uniformly hyperbolic structure. As mentioned above, one main obstruction to such structure in our setting is the non-uniform versions of expansiveness. Precisely, we consider the following two types of weakened expansiveness:
\begin{enumerate} [label=\upshape{(E\arabic{*})}]
     \item \label{E1} $(X,F)$ is entropy expansive,
     \item \label{E2} $(X,F)$ is not entropy expansive, yet $h_{\exp}^{\perp}<h$.
 \end{enumerate}
Here $h:=h_{\text{top}}(F)$ denotes the topological entropy of the flow and $h_{\exp}^{\perp}$ the obstruction entropy to expansiveness (see Definition \ref{def:entropyexpansive}). Roughly speaking, $h_{\exp}^{\perp}$ measures the entropy for the non-expansive part of the system, and the condition $h_{\exp}^{\perp}<h$ indicates that all $\mu\in \mathcal M_F(X)$ with large measure-theoretic entropy $h_\mu(F)$ is almost expansive. 

We equip $C(X,\mathbb{R})$ with the usual $C^0$-norm given by
$$\|\varphi\|_{C^0}:=\max\{|\varphi(x)|: x\in X\} \quad \text{ for each }\varphi\in C(X,\RR).$$
For any $H\geq 0$, let 
$$
\mathcal{R}_H:=\{\ph\in C(X,\RR): h_\mu(F)\le H \ \text{for all}\  \mu\in \mathcal{M}_{\max}(\ph)\}.
$$
Our first result is stated as follows.
\begin{thma} \label{thm:main} 
Let $(X,F)$ be a fixed-points free continuous flow on a compact metric space. Suppose that $(X,F)$ has expansive type of \ref{E1} or \ref{E2}, and $X\times\mathbb{R}^+$ admits a decomposition $(\mathcal{P},\mathcal{G},\mathcal{S})$ such that $\mathcal{G}$ satisfies weak controlled specification at all scale $\eta>0$, with gap function $h_{\eta}^{\mathcal{G}}$ satisfying $\liminf_{t\to \infty}h_{\eta}^{\mathcal{G}}(t)/\log t=0$, and $h([\mathcal{P}]\cup [\mathcal{S}])<h$. Write
\begin{align*}
H^{\perp}:=\left\{
\begin{array}{cr}
h([\mathcal{P}]\cup [\mathcal{S}]), & \text{if } (X,F) \text{ is of expansive type \ref{E1}},
\\
\max\{h([\mathcal{P}]\cup [\mathcal{S}]),h_{\exp}^{\perp}\}, & \text{if } (X,F) \text{ is of expansive type \ref{E2}}.
\end{array}
\right.
\end{align*}
Then $R_{H}$ is dense $G_\delta$ in $C(X,\RR)$ for all $H\in [H^{\perp},h)$. 
\end{thma}

Definitions for orbit decomposition and weak controlled specification can be found in Definitions \ref{def:decomp} and \ref{speci} respectively. Compared to uniformly hyperbolic settings in which typically periodic optimization concerning periodicity of maximizing measures for generic Lipschitz or H\"older functions are obtained, Theorem \ref{thm:main} reveals that generic continuous functions have ``small-entropy'' maximization for many general non-uniformly hyperbolic systems, which parallels results in \cite{STY} for non-Markov shifts and partially generalizes of Morris's \cite{Morris2} result for systems with Bowen's specification property. 

As generic continuous functions have unique maximizing measure (see Proposition \ref{thm:con}), writing 
$$
\mathcal{U}_H:=\{\ph\in C(X,\RR): h_\mu(F)\le H \ \text{for the unique}\  \mu\in \mathcal{M}_{\max}(\ph)\}
$$ 
for all $H\geq 0$, we get the following immediately.  
\begin{thma} \label{coro:uniquedenseGdelta} 
 Under the setting of Theorem \ref{thm:main}, $\mathcal{U}_{H}$ is dense $G_\delta$ in $C(X,\RR)$ for all $H\in [H^{\perp},h)$.
\end{thma}

The next theorem demonstrates continuous functions with maximizing measures of ``relatively large'' entropy are not isolated, but persistent and abundant. Moreover, such continuous functions' optimization is complex, characterized by uncountably many ergodic measures. 
\begin{thma} \label{thm:pathelogic}
    Suppose that $h_{\exp}^{\perp}<h$, and $X\times\mathbb{R}^+$ admits a decomposition $(\mathcal{P},\mathcal{G},\mathcal{S})$ as in Theorem \ref{thm:main}. Assume also that every Lipschitz function in $C(X,\mathbb{R})$ satisfies Bowen property along $\mathcal{G}$. Then for any $H\in [H^{\perp},h)$, $\ph\in C(X,\RR)\setminus \mathcal{R}_H$ and any neighborhood $U$ of $\ph$ in $C(X,\RR)$, there exists $\psi\in U$ such that the set $\{\mu\in \mathcal{M}_{\max}(\psi): h_\mu(F)>H\}$ contains uncountably many ergodic measures.
\end{thma}

The readers may refer to \cite[Theorem A]{Shi18} and \cite[Theorem B]{STY} for the symbolic analog of Theorem \ref{thm:pathelogic}. In the proof of Theorem \ref{thm:pathelogic}, we use convex analysis approximation results to get desired continuous function and its corresponding maximizing measure. Since low temperature equilibrium states almost maximize the potentials, maximizing measure can be found in a small neighborhood of such equilibrium state. Abundance of equilibrium states allows to locate uncountably many  maximizing measures. The general technical result leading towards Theorem \ref{thm:pathelogic} has spirit in the following sense.

\begin{thma}(Ground States) \label{thm:groundstates}
Let $(X,F)$ be as in Theorem \ref{thm:pathelogic}. Then for each $\mu\in \mathcal{M}_F^e(X)$ with $h_\mu(F) > H^{\perp}$, the following statements hold: For any constant $H_0\in (H^{\perp},h_{\mu}(F))$, any open subset $U$ of $\mathcal{M}_F(X)$ that contains $\mu$, there exists a Lipschitz
continuous function $\ph: X\to \RR$ such that:
\begin{enumerate}
    \item For any $\alpha\ge 0$, there exists a unique equilibrium state for the potential $\alpha\ph$,  denoted by $\mu_{\alpha \ph}$.
\item  The map $\alpha\mapsto \mu_{\alpha \ph}\in \mathcal{M}_F(X)$ is continuous.
\item For all sufficiently large $\alpha>0$ we have $\mu_{\alpha \ph}\in U$, and 
$$\lim_{\alpha\to \infty}h_{\mu_{\alpha \ph}}(F) > H_0.$$
\item There exists a pair of constants $0<\alpha_1<\alpha_2<\infty$ such that the map $\alpha\mapsto \mu_{\alpha \ph}$
is injective on $(\alpha_1,\alpha_2)$, and $h_{\mu_{\alpha \ph}}(F)>H_0$ for all $\alpha\in (\alpha_1,\alpha_2)$.
\end{enumerate} 
\end{thma}

\subsection{Application to geodesic flows}

As an application, we consider the geodesic flow over a closed rank one manifold of nonpositive curvature. 

Let $M$ be a closed rank one manifold of nonpositive curvature and $SM$ its unit tangent bundle. Let $G=\{g^t\}_{t\in \RR}:SM\to SM$ denote the geodesic flow. By Knieper \cite{Kn2} and Burns-Climenhaga-Fisher-Thompson \cite{BCFT}, we have $h_{\text{top}}(\text{Sing})<h_{\text{top}}(G)$, where $\text{Sing}$ denotes the singular set of the geodesic flow.

\begin{thma}\label{thm:bcft}
Let $M$ be a closed rank one manifold of nonpositive curvature, and $G=\{g^t\}_{t\in \RR}:SM\to SM$ the geodesic flow. Let $H\in [h_{\text{top}}(\text{Sing}), h_{\text{top}}(G))$. Then the following genericity results hold: 
\begin{enumerate}
\item The set $\mathcal{R}_H:=\{\ph\in C(SM,\RR): h_\mu(G)\le H \ \text{for all}\  \mu\in \mathcal{M}_{\max}(\ph)\}$ is dense $G_\delta$ in $C(SM,\RR)$.
\item For any $\ph\in C(SM,\RR)\setminus \mathcal{R}_H$ and any neighborhood $U$ of $\ph$ in $C(SM,\RR)$, there exists $\psi\in U$ such that the set 
$\{\mu\in \mathcal{M}_{\max}(\psi)\}$ contains uncountably many, fully supported ergodic measures.
\end{enumerate}
Furthermore, let $\mu\in \mathcal{M}_G^e(SM)$ with $h_\mu(G) > H$. Then for any open subset $U$ of $\mathcal{M}_G(SM)$ that contains $\mu$, there exists a Lipschitz continuous function $\ph: SM\to \RR$ such that:
\begin{enumerate}
    \item For any $\beta\ge 0$, there exists a unique equilibrium state for the potential $\beta\ph$, which has the weak Gibbs property, denoted by $\mu_{\beta \ph}$.
\item  The map $\beta\mapsto \mu_{\beta \ph}\in \mathcal{M}_G(SM)$ is continuous and injective.
\item For all sufficiently large $\beta>0$ we have $\mu_{\beta \ph}\in U$, and there exists $H_0 > H$ such that $\lim_{\beta\to \infty}h_{\mu_{\beta \ph}}(F) = H_0$.
\end{enumerate}
\end{thma}

Combining with Proposition \ref{thm:con}, we have
\begin{thma}
 Let $M$ be a closed rank one manifold of nonpositive curvature, and $G=\{g^t\}_{t\in \RR}:SM\to SM$ the geodesic flow. 
 Then for any $H\in [h_{top}(\text{Sing}), h_{\text{top}}(G))$, the set
 $$\mathcal{U}_H:=\{\ph\in C(SM,\RR): h_\mu(G)\le H \ \text{for the unique}\  \mu\in \mathcal{M}_{\max}(\ph)\}$$ 
 is dense $G_\delta$ in $C(SM,\RR)$.   
\end{thma}

We also consider the frame flow over a compact nonpositively curved rank one Riemannian manifold. The frame flow is a group extension of the geodeisc flow, which can be thought as a system with mixture of nonuniform hyperbolicity and partial hyperbolicity. If the manifold satisfies a type of bunched condition on curvature, we are also able to obtain results displayed in Theorem \ref{thm:bcft}. 
\begin{thma} \label{thm:frame} 
Let $M$ be a closed, oriented rank one $n$-manifold with bunched nonpositive curvature. Suppose that the frame flow $F=\{F^t\}_{t\in \RR}:FM\to FM$ is topologically transitive. Let $H\in [h_{\text{top}}(\text{Sing}), h_{\text{top}}(G))$. Then the following genericity results hold:
\begin{enumerate}
\item The set $\mathcal{R}_H:=\{\ph\in C(FM,\RR): h_\mu(F)\le H \ \text{for all}\  \mu\in \mathcal{M}_{\max}(\ph)\}$ is dense $G_\delta$ in $C(FM,\RR)$.
\item For any $\ph\in C(FM,\RR)\setminus \mathcal{R}_H$ and any neighborhood $U$ of $\ph$ in $C(FM,\RR)$, there exists $\psi\in U$ such that the set 
$\{\mu\in \mathcal{M}_{\max}(\psi)\}$ contains uncountably many fully supported ergodic measures.
\end{enumerate}
Furthermore, let $\mu\in \mathcal{M}_F^e(FM)$ with $h_\mu(F) > H$. Then for any open subset $U$ of $\mathcal{M}_F(FM)$ that contains $\mu$, there exists a Lipschitz continuous function $\ph: FM\to \RR$ such that:
\begin{enumerate}
    \item[(3)] For any $\beta\ge 0$, there exists a unique equilibrium state for the potential $\beta\ph$, denoted by $\mu_{\beta \ph}$.
\item[(4)]  The map $\beta\mapsto \mu_{\beta \ph}\in \mathcal{M}_F(FM)$ is continuous and injective.
\item[(5)] For all sufficiently large $\beta>0$ we have $\mu_{\beta \ph}\in U$, and there exists $H_0 > H$ such that $\lim_{\beta\to \infty}h_{\mu_{\beta \ph}}(F) = H_0$.
\end{enumerate}
\end{thma}

The main obstacle to state Theorem \ref{thm:frame} beyond specific examples is that in addition to have enough control of distortion for a Lipschitz function along $\mathcal{G}$, one needs to guarantee that the pressure for such functions to be reflected at some fixed scale. This is in general not always true, even in the entropy-expansive case. In the meantime, though it is generally not known that whether the equilibrium states derived in \cite{CT16} or \cite{WangWu} are fully supported, the arguments in these works provides us with the lower Gibbs property for the Bowen balls associated to orbit segments in $\mathcal{G}$, which in turn indicates that those measures in Theorem \ref{bcftthm}  and Theorem \ref{thm:frame} (2) are indeed fully supported due to the additional dynamical properties brought by their corresponding geometric structures.

\subsection{Some history}
For expanding transformations, Bousch \cite{Bo01} proves that for generic continuous functions the maximizing measure is fully supported. Bousch's theorem in the hyperbolic case is presented by Jenkinson in \cite{Jen06b}. 
Not only does this contrast with typically periodic optimization, but also it contrasts with intuition. Indeed an open problem is to exhibit constructively a continuous function over an expanding or Anosov system such that the unique maximizing measure is fully supported. Br\'emont  \cite{Bre08} shows that typical maximizing measures have zero entropy if the dynamical system has the property that measures supported on periodic orbits are dense in $\mathcal{M}_F(X)$ and if the entropy map is upper semi-continuous.

Morris \cite{Morris2} united the results of Bousch \cite{Bo01}, Jenkinson \cite{Jen06b} and Brémont  \cite{Bre08}, proving that for a system with the Bowen specification property and for a generic continuous function, the maximizing measure is unique, fully supported and has zero entropy. In contrast to Morris’ result, Shinoda \cite{Shi18} proved that for a dense set of continuous functions on a topologically mixing Markov shift, there exist uncountably many, fully supported ergodic maximizing measures with positive entropy. Morris' result \cite{Morris2}  and Shinoda's result \cite{Shi18} are unified and generalized to a wide class including non-Markov shifts in \cite{STY}.

Although rich structure of ergodic optimization has been demonstrated for continuous functions, typicality of periodic optimization can still be expected for more regular observables \cite{Con}. Recently, Huang et al. \cite{HLMXZ} have resolved the TPO conjecture for a broad class of dynamical systems including Axiom A attractors, Anosov diffeomorphisms and uniformly
expanding maps, and for observables with H\"older, Lipschitz and $C^1$ regularity respectively. For results toward a resolution of  TPO conjecture in  various settings, see \cite{Bo01, Con, Jen19,HLMXZ, GS1,GS2} and the references therein.

\subsection{Organization of the paper}
This paper is structured as follows. In Section \ref{Preliminaries}, we recall some fundamental concepts and results from topological dynamics and thermodynamic formalism. Particularly, we state orbit decomposition $(\mathcal{P},\mathcal{G},\mathcal{S})$ and the crucial thermodynamic results of Climenhaga–Thompson regarding  uniqueness of equilibrium states. Generic properties for maximizing measures for flows in the space of continuous functions have also been established. In Section \ref{Proofs for Theorem}, two key lemmas concerning entropy contribution of orbit segments with long bad prefixes/suffixes and remaining good segments are laid out first. The proof of Theorem \ref{thm:main} heavily relies on Proposition \ref{prop:entropysmalldense} where specification is applied on a carefully selected good orbit segment to construct a compact invariant set with low complexity. The proof of Theorem \ref{thm:groundstates} shares similar strategy. A compact invariant subset with desired entropy control (see Proposition \ref{prop:entropyofY1}) is built and the rest follows from thermodynamic results. The proof of Theorem \ref{thm:pathelogic} finalizes Section \ref{Proofs for Theorem}.
In Section \ref{Examples and applications}, we apply these results to geodesic flows and frame flows on rank one manifolds in nonpositive curvature. Technical proofs of some results for flows are provided in the Appendix for self-containedness.

\section{Preliminaries}\label{Preliminaries}

\subsection{Backgrounds on thermodynamical formalism}
Let $F=(f_t)_{t\in \mathbb{R}}$ be a continuous flow on a compact metric space $(X,d)$, and $\mathcal{M}_F(X)$ (resp. $\mathcal{M}_F^e(X))$ denote the set of all invariant (resp. ergodic) probability measures on $X$ under $F$. Given any $t>0$, the $t$-th Bowen metric is defined by 
 $$d_t(x,y)=d_t^F(x,y):=\max_{0\le s\le t}d(f_sx, f_sy), \quad \forall x,y\in X.$$
 The $t$-th Bowen ball is then the ball with respect to $d_t$. Precisely, given $\eps,t>0$, the $t$-th Bowen ball centered at $x\in X$ with radius $\eps$ is defined by
 $$B_t(x,\eps):=\{y\in X: d_t(x,y)<\eps\}.$$

A set $A\subset X$ is called $(t,\eps)$-separated, if for any distinct $x,y\in A$, $d_t(x,y)>\eps$. For any $\eps>0$, let $\Lambda_t(X,\eps)$ be the cardinality of maximally $(t,\eps)$-separated set for $X$. Topological entropy can be defined through exponential growth rate of cardinality of separated sets as follows,
$$
h_{\text{top}}(F) = \lim_{\eps \to 0} \limsup_{t \to \infty} \frac{1}{t} \log \Lambda_t(X, \eps).
$$

\subsubsection{Pressure and orbit decomposition}
Denote by $X\times \mathbb{R}^+$ the space of finite orbit segments for $(X,F)$. For any $\mathcal{C}\subset X\times \mathbb{R}^+$ and $t>0$, we write
\begin{equation*} 
    \mathcal{C}_t:=\{x\in X \ | \ (x,t)\in \mathcal{C}\}.
\end{equation*} 
Given a potential $\ph: X\to \mathbb{R}$ and time $t>0$, write 
$$
\Phi(x,t):=\int_0^t \ph(f_sx)ds.
$$
\begin{definition}
Given $\mathcal{C}\subset X\times \mathbb{R}^+$ and $t,\eps>0$, the (separated) partition function of $\ph$ is given by
$$
\Lambda_t(\mathcal{C},\ph,\eps):=\sup\Big\{\sum_{x\in E}e^{\Phi(x,t)}\ |\ E\subset \mathcal{C}_t \text{ is } (t,\eps)\text{-separated}\Big\}.
$$
The \emph{pressure of $\ph$ over $\mathcal{C}$ at scale $\eps$} is given by
$$
P(\mathcal{C},\ph,\eps)=\limsup_{t\to \infty}\frac{1}{t}\log \Lambda_t(\mathcal{C},\ph,\eps),
$$ 
then we define the \emph{pressure of $\ph$ over $\mathcal{C}$} as
$$
P(\mathcal{C},\ph)=\limsup_{\eps\to 0}P(\mathcal{C},\ph,\eps).
$$ 
\end{definition}
\begin{remark}
\begin{enumerate}
    \item If $\mathcal{C}=X\times \mathbb{R}^+$, we write $P(\varphi,\eps):=P(X\times \mathbb{R}^+,\varphi,\eps)$ and the \emph{pressure of $\ph$} is defined as 
$P(\varphi):=\lim\limits_{\eps\to 0}P(\varphi,\eps).$
\item When $\varphi=0$, we get the topological entropy
$$h(\mathcal{C},\eps):=P(\mathcal{C},0,\eps), \quad h(\mathcal{C}):=P(\mathcal{C},0), \quad h_{\text{top}}(F):= P(0).$$
For any subset $Y\subset X$, we also write $h(Y):=h(Y\times \mathbb{R}^+)$.
\end{enumerate}
\end{remark}

It is clear from the definition of $P(\mathcal{C},\ph)$ that $\forall \delta,\eps>0$, there exists some constant $C_{\delta}(\mathcal{C}, \eps)>0$ such that   
\begin{equation}\label{upper}
\Lambda_t(\mathcal{C},\ph,\eps)<C_{\delta}(\mathcal{C}, \eps)e^{t(P(\mathcal C, \ph)+\delta)}, \quad \forall t>0.   
\end{equation}
This estimation will be used in different scenarios in the rest of the paper.

\begin{definition} \label{def:decomp}
    A \emph{decomposition} $(\mathcal{P},\mathcal{G},\mathcal{S})$ for $\mathcal{D}\subset X\times \mathbb{R}^+$ consists of three collections $\mathcal{P},\mathcal{G},\mathcal{S} \subset X\times \mathbb{R}^+$ and functions $p,g,s:\mathcal{D}\to \mathbb{R}^+$ such that for every $(x,t)\in \mathcal{D}$, writing $p(x,t), g(x,t), s(x,t)$ as $p,g,s$ respectively, we have $t=p+g+s$, and
    $$
(x,p)\in \mathcal{P}, \qquad (f_p(x),g)\in \mathcal{G}, \qquad (f_{p+g}(x),s)\in \mathcal{S}.
    $$
    We always assume that $X\times \{0\}$ belongs to all collections $(\mathcal{P},\mathcal{G},\mathcal{S})$ by default.
\end{definition}

As in \cite[(2.9)]{CT16}, due to the frequent usage of discretization argument throughout the paper, given $\mathcal{C}\subset X\times \mathbb{R}^+$, we often deal with the following slightly larger collection of orbit segments
$$
[\mathcal{C}]:=\{(x,n)\in X\times \mathbb{N} \ | \ (f_{-s_1}(x),n+s_1+s_2)\in \mathcal{C} \ \text{ for some }s_1,s_2\in [0,1] \}.
$$
In this paper, we always assume that $X\times \mathbb{R}^+$ itself admits a decomposition $(\mathcal{P},\mathcal{G},\mathcal{S})$. Then for each element in $X\times \mathbb{R}^+$, by chopping off the head and tail corresponding to $[\mathcal{P}]$ and $[\mathcal{S}]$ respectively, the remainder is denoted by $\mathcal{G}^1$, which is defined as follows
$$
 \mathcal{G}^1:=\{(x,t)\in X\times \mathbb{R}^+ \ | \ (f_{s_1}(x),t-s_1-s_2)\in \mathcal{G} \ \text{ for some }s_1,s_2\in [0,1] \}
$$
\subsubsection{Distortions, specification and expansiveness}

Let $g$ be a positive function on $\mathbb{R}^+$, and $\varphi$ a potential on $X$. For any $\eps>0$ and $\mathcal{C}\in X\times \mathbb{R}^+$, we say $\varphi$ is $g$-distorted over $\mathcal{C}$ at scale $\eps$ if for any $(x,t)\in \mathcal{C}$, we have
$$
|\Phi(x,t)-\Phi(y,t)|\leq g(t) \quad \text{ for all } y\in B_t(x,\eps).
$$
Moreover, if $g$ is bounded from above, then we say $\varphi$ satisfies \emph{Bowen property} over $\mathcal{C}$ at scale $\eps$. We also say $\varphi$ satisfies Bowen property over $\mathcal{C}$ if it does at some scale.

\begin{definition}\label{speci}
Let $h$ be a non-decreasing function defined on $\mathbb{R}^+$, and $\eta>0$ be a constant. We say $\mathcal{C}\subset X\times \mathbb{R}^+$ satisfies \emph{weak controlled specification property} with gap function $h_{\eta}^{\mathcal{C}}=h$ at scale $\eta$ if for every $n\geq 2$ and $((x_i,t_i))_{i=1}^n\subset \mathcal{C}$, there exist
\begin{enumerate}
    \item a sequence $(\tau_i)_{i=1}^{n-1}$ with $\tau_i\leq \max\{h_{\eta}^{\mathcal{C}}(t_i),h_{\eta}^{\mathcal{C}}(t_{i+1})\}$ for each $i$,
    \item a point $y\in X$, for which we denoted by $\text{Spec}^{n,\eta}_{(\tau_i)}(((x_i,t_i)))$
\end{enumerate}
such that
$$
d_{t_i}(f_{\sum_{j=1}^{i-1}(t_j+\tau_j)}(y),x_i)<\eta, \quad \forall 1\le i\le n.
$$

If we can choose $h_{\eta}^{\mathcal{C}}$ to be a constant $\tau>0$, then we say $\mathcal{C}$ satisfies \emph{weak specification property} with gap $\tau$ at scale $\eta$. We say $\mathcal{C}$ satisfies \emph{weak specification property} if it satisfies weak specification property at every scale $\eta>0$.
\end{definition}

Throughout the paper, the gap function $h_{\eta}^{\mathcal{C}}$ is always assumed to be greater than $1$, for simplicity of stating some conditions.

For every $x\in X$ and $\eps>0$, consider the corresponding two-sided infinite Bowen ball given by
$$
\Gamma_{\eps}(x):=\{y\in X:d(f_t(x),f_t(y))\leq\eps \text{ for all }t\in \mathbb{R}\}.
$$
The size of $\{\Gamma_{\eps}(x)\}_{x\in X}$ is frequently used in characterizing non-expansiveness of $(X,F)$. Precisely, we consider two versions of weakened expansiveness in this paper. The first is known as ``entropy expansiveness'', which concerns the topological entropy of $\Gamma_{\eps}(x)$ and is defined as follows.
\begin{definition}
    Given $\eps>0$, we say $(X,F)$ is \emph{entropy expansive at scale $\eps$} if for every $x\in X$, we have $h(\Gamma_{\eps}(x))=0$. $(X,F)$ is \emph{entropy expansive} if it is so at some scale. 
\end{definition}

The second is defined using so-called obstruction to expansiveness, which is a measure-theoretic analog of expansiveness in asking for almost expansiveness for all measures with large entropy. Given $\eps>0$, let the set of \emph{non-expansive points at scale $\eps$} for $F$ be 
$$
\text{NE}(\eps):=\{x\in X:\Gamma_{\eps}(x)\nsubseteq f_{[-s,s]}(x)\text{ for any }s>0\}.
$$
The following quantity introduced in \cite[$\mathsection 2.5$]{CT16} captures the largest information of a non-expansive ergodic measure.
\begin{definition}\label{def:entropyexpansive}
    Given a potential function $\ph\in C(X, \mathbb{R})$ and $\eps>0$, the \emph{pressure of obstructions to expansiveness at scale $\eps$} is defined as
    $$
P^{\perp}_{\exp}(\ph,\eps):=\sup_{\mu\in \mathcal{M}_F^e(X)}\Bigl \{h_{\mu}(f_1)+\int \ph d\mu:\mu(\text{NE}(\eps))=1  \Bigl \},
    $$
    and
    $$
P^{\perp}_{\exp}(\ph)=\lim_{\eps\to 0}P^{\perp}_{\exp}(\ph,\eps).
    $$
When $\ph=0$, we write $h_{\exp}^{\perp}:=P^{\perp}_{\exp}(0)$.
\end{definition}

Notice that if $\mu\in \mathcal{M}_F^e(X)$ satisfies $h_{\mu}(f_1)+\int \ph d\mu>P^{\perp}_{\exp}(\ph,\eps)$, then $\mu(\text{NE}(\eps))=0$. In other words, $\mu$ is \emph{almost expansive} at scale $\eps$. In particular, it follows from the continuity of the flow and compactness of $X$ that $\mu$ is \emph{almost entropy expansive}, meaning that $h(\Gamma_{\eps}(x))=0$ for $\mu$-a.e. $x$. It is also clear that when $(X,F)$ is entropy expansive, every $\mu\in \mathcal{M}_F(X)$ is almost entropy expansive.

\subsubsection{Uniqueness of equilibrium states}
We are ready to state Climenhaga and Thompson's theorem from \cite{CT16}.
\begin{theorem}(\cite[Theorem A]{CT16}) \label{thm:CT}
  Let $(X, F)$ be a continuous flow on a compact metric space, and $\varphi: X\to \RR$ a continuous potential function. Suppose that $P^{\perp}_{\exp}(\varphi) <P(\varphi)$ and that $X\times \RR^+$ admits 
a decomposition $(\PPP, \GGG, \SSS)$ with the following properties:
\begin{enumerate}
    \item $\GGG$ has the weak specification property;
    \item $\varphi$ has the Bowen property on $\GGG$; 
    \item  $P([\PPP]\cup [\SSS], \varphi)<P(\varphi).$
\end{enumerate}
  Then $\varphi$ has a unique equilibrium state.     
\end{theorem}

Observe that the condition $P^{\perp}_{\exp}(\varphi) <P(\varphi)$ will generally not hold when the system admits a neutral direction that is not along the flow. To address this issue and make Climenhaga-Thompson's criteria still applicable, motivated by Pavlov's work in \cite{Pav19}\footnote{The idea traces back to one of Climenhaga’s math blogs online; see
https://vaughnclimenhaga.wordpress.com/2015/06/12/unique-mmes-with-specification-an-alternate-proof/}, the second \& third author \cite{WangWu} propose a general criteria in characterizing thermodynamic formalism for entropy-expansive systems assumed with a more delicately controlled specification. 

\begin{theorem}(\cite[Theorem A]{WangWu})\label{WW}
 Let $(X,F)$ be a continuous flow on a compact metric space, and $\ph:X\to \mathbb{R}$ be a continuous potential. Suppose there exists some $\eps>0$ at which $(X,F)$ is entropy expansive, and  $X\times\mathbb{R}^+$ admits a decomposition $(\mathcal{P},\mathcal{G},\mathcal{S})$ satisfying the following conditions:
    \begin{enumerate}
        \item For every $\eta>0$, $\mathcal{G}$ satisfies weak controlled specification at scale $\eta$ with gap function $h_{\eta}^{\mathcal{G}}$;
        \item $\ph$ is $g$-distorted over $X\times \mathbb{R}^+$ at scale $\eps$ (see \cite[$\mathsection 2.2$]{WangWu});
        \item For every $\eta>0$, $h_{\eta}$ and $g$ satisfy $\liminf_{t\to \infty}\frac{h_{\eta}^{\mathcal{G}}(t)+g(t)}{\log t}=0$;
        \item $\lim_{\eta\to 0}P([\mathcal{P}]\cup [\mathcal{S}],\ph,\eta,\eps)<P(\ph)$.
    \end{enumerate}
    Then $(X,F,\ph)$ has a unique equilibrium state.    
\end{theorem}

\begin{theorem}(\cite[Theorem C]{WangWu})\label{thm:WW1}
 Let $(X,F)$ be a continuous flow on a compact metric space, and $\ph:X\to \mathbb{R}$ be a continuous potential. Suppose that $P_{\exp}^\perp(\ph)<P(\ph)$, and $X\times\mathbb{R}^+$ which admits a decomposition $(\mathcal{P},\mathcal{G},\mathcal{S})$ satisfying conditions (1), (3) and (4) in Theorem \ref{WW}, as well as
 $$
\ph \text{ is }g\text{-distorted over }\mathcal{G} \text{ at scale }\eps.
 $$
    Then $(X,F,\ph)$ has a unique equilibrium state.    
\end{theorem} 

Both theorems are immediate applications of \cite[Theorem D]{WangWu}. 

\begin{theorem} (\cite[Theorem D]{WangWu})\label{thm:weprove}
     Let $(X,F)$ be a continuous flow on a compact metric space, and $\ph:X\to \mathbb{R}$ a continuous potential. Suppose there exists some $\eps>0$ 
     such that
     $$
P(\ph,\eta)=P(\ph) \text{ for all }\eta\in (0,\eps/2)
     $$
     and for every $t>0$, any finite partition $\mathcal{A}_t$ with $\text{Diam}_t(\mathcal{A}_t)<\eps$, and any $\mu\in \mathcal{M}_F^e(X)$ being almost entropy expansive at scale $\eps$, we have
     $$
h_{\mu}(f_t,\mathcal{A}_t)=h_{\mu}(f_t).
     $$
     Assume further that there exists some $\mathcal{D}\subset X\times\mathbb{R}^+$ admitting a decomposition $(\mathcal{P},\mathcal{G},\mathcal{S})$ satisfying the following conditions:
    \begin{enumerate}
        \item For every $\eta>0$, $\mathcal{G}$ satisfies weak controlled specification at scale $\eta$ with gap function $h_{\eta}^{\mathcal{G}}$;
        \item $\ph$ is $g$-distorted over $\mathcal{G}$ at scale $\eps$;
        \item For every $\eta>0$, $h_{\eta}$ and $g$ satisfy $\liminf_{t\to \infty}\frac{h_{\eta}^{\mathcal{G}}(t)+g(t)}{\log t}=0$;
        \item $\lim_{\eta\to 0}P(\mathcal{D}^c\cup [\mathcal{P}]\cup [\mathcal{S}],\ph,\eta,\eps)<P(\ph)$.
    \end{enumerate}
    Then $(X,F,\ph)$ has a unique equilibrium state.
\end{theorem}
\begin{remark} \label{rmk:heta}
    As mentioned in \cite[Remark 2.6]{WangWu}, weak controlled specification over $\mathcal{G}$ can be safely transferred to which over $\mathcal{G}^1$. Precisely, when $\mathcal{G}$ satisfies weak controlled specification at scale $\eta$ with gap function $h_{\eta}^{\mathcal{G}}$, so does $\mathcal{G}^1$ at scale $\eta'$ with gap function $h_{\eta}^{\mathcal{G}}$, where 
    $$\eta':=\max\{d(f_s(x),f_s(y)):d(x,y)\leq \eta, s\in [0,1]\}.$$ In particular, $\mathcal{G}$ satisfying weak controlled specification at all scales is equivalent to which for $\mathcal{G}^1$. In this case, as we always implement weak controlled specification over $\mathcal{G}^1$ throughout the paper, we will abbreviate  $h_{\eta}^{\mathcal{G}^1}$, the gap function over $\mathcal{G}^1$ at scale $\eta$, as $h_{\eta}$. It is also clear that a similar equivalence holds for Bowen property of $\varphi$ over $\mathcal{G}$ and $\mathcal{G}^1$, and we refer the interested readers to \cite[Remark 2.6]{WangWu} without further explanations as that is indeed simpler.
\end{remark}

\subsection{Lemmas that come in handy}

\begin{lemma}[Bi-infinite specification] \label{lem:infinitespecification}
    Suppose that $\mathcal{G}^1$ satisfies weak controlled specification at scale $\eta$ with gap function $h_{\eta}$. Then for any bi-infinite sequence of orbit segments $\{(x_i,t_i)\}_{i=-\infty}^{\infty}$ from $\mathcal{G}^1$ and $\eta>0$, there exist $y\in X$ and a sequence $\{h_i\}_{i=-\infty}^{\infty}$ with $h_i\in [0,\max\{h_\eta(t_i),h_\eta(t_{i+1})\}]$ such that the following holds true: writing $m_0=0$, $m_i:=\sum_{j=0}^{i-1}(t_j+h_j)$, and $m_{-i}:=-\sum_{j=-i}^{-1}(t_j+h_j)$ for all $i\in \mathbb{Z}$, we must have $d_{t_{i}}(x_{i},f_{m_i}(y))\leq \eta$.
\end{lemma}
\begin{proof}
The lemma follows immediately from Definition \ref{speci} of weak controlled specification and compactness of $X$, by a standard diagonal argument.
\end{proof}

The following lemma aims at providing enough upper semi-continuity of the entropy map in the case of $h_{\exp}^{\perp}<h$.
\begin{lemma} \label{lem:usc}
    Let $\eps>0$ and $\{\mu_n\}_{n\in \mathbb{N}}$ be a sequence of elements in $\mathcal{M}_{F}(X)$ such that 
    \begin{itemize}
        \item each $\mu_n$ is almost expansive at scale $\eps$,
        \item $\mu_n\to \mu$ as $n\to \infty$ in the weak*-topology.
    \end{itemize}
    Then $\limsup_{n\to \infty}h_{\mu_n}(F)\leq h_{\mu}(F)$.
\end{lemma}
\begin{proof}
    Given any partition $\zeta$ for $X$ such that $\text{diam}(\zeta)\leq \eps$ in the metric $d_1$, by \cite[Proposition 3.4]{CT16}, we have $h_{\mu_n}(F,\zeta)=h_{\mu_n}(F)$ for all $n$. Given any $\gamma>0$, and any partition $\xi$ with $\text{diam}(\xi)\leq \eps$ in the metric $d_1$, let $N\in \mathbb{N}$ be such that
    $$
\frac{1}{N}H_{\mu}(\bigvee_{j=0}^{N-1}f_1^{-j}\xi)<h_{\mu}(f_1,\xi)+\frac{\gamma}{2}.
    $$
 By following the argument in the proof of \cite[Theorem 8.2]{Wal82}, we can construct a corresponding partition $\xi'=\xi'(\xi,N)$ satisfying $\text{diam}(\xi')\leq \eps$ and $\mu(\partial \xi')=0$ such that
    $$
\limsup_{n\to \infty}h_{\mu_n}(F)=\limsup_{n\to \infty}h_{\mu_n}(f_1,\xi')<\frac{1}{N}H_{\mu}(\bigvee_{j=0}^{N-1}f_1^{-j}\xi)+\frac{\gamma}{2}<h_{\mu}(f_1,\xi)+\gamma,
    $$
    which concludes the lemma by making $\gamma\to 0$.
\end{proof}
As an immediate corollary, we have
\begin{corollary} \label{coro:entropyusc}
    Let $\{\varphi_n\}_{n\in \mathbb{N}}$ be a sequence of potentials satisfying 
    $$P_{\exp}^{\perp}(\varphi_n,\eps)<P(\varphi_n) \text{\ for some\ } \eps>0.$$ 
    If $(\mu_n)_{n\in \mathbb{N}}$ is a sequence of equilibrium states for $\ph_n$ respectively, that converges to some $\mu\in \mathcal{M}_F(X)$ as $n\to \infty$ in weak*-topology. Then 
    $$\limsup_{n\to \infty}h_{\mu_n}(F)\leq h_{\mu}(F).$$
\end{corollary}
In the mean time, we have
\begin{corollary} \label{coro:Uhopen}
    If $h_{\exp}^{\perp}<h$, given any $H\in (h_{\exp}^{\perp},h)$, the set 
    $$U_h:=\{\mu\in \overline{\mathcal{M}_F^e(X)}:h_{\mu}(F)\in [0,H)\}$$ 
    is an open subset of $\overline{\mathcal{M}_F^e(X)}$.
\end{corollary}
\begin{proof}
Let $\eps>0$ be such that $h_{\exp}^{\perp}(\eps)<H$. It suffices to show that $\overline{\mathcal{M}_F^e(X)}\setminus U_h$ is closed in $\overline{\mathcal{M}_F^e(X)}$. To see this, for any sequence $\{\mu_n\}_{n\in \mathbb{N}}$ in $\overline{\mathcal{M}_F^e(X)}\setminus U_h$ such that $\mu_n\to \mu$ as $n\to \infty$ in the weak*-topology, by noticing that $h_{\mu_n}(F)\geq H>h_{\exp}^{\perp}(\eps)$, we know each $\mu_n$ is almost expansive at scale $\eps$, which by Lemma \ref{lem:usc} implies that $h_{\mu}(F)\geq H$. It is also clear that $\mu\in \overline{\mathcal{M}_F^e(X)}$. Consequently, we know $\mu\in \overline{\mathcal{M}_F^e(X)}\setminus U_h$, concluding $U_h$ being open in $\overline{\mathcal{M}_F^e(X)}$.
\end{proof}

\subsection{Limit behavior of equilibrium states}
    For a dynamical system $(X,F)$ with potential $\ph\in C(X,\mathbb{R})$, if $\alpha \ph$ has an equilibrium state $\mu_{\alpha \ph}$, we call any weak-*limit of $(\mu_{\alpha_n \ph})_{n\in \mathbb{N}}$ with $\alpha_n\uparrow \infty$ a \emph{ground state} for $(X,F,\ph)$. Such limit behavior has rich physical intuitions behind its name. In fact, as $\alpha$ is often interpreted as the inverse of temperature, the process of making $\alpha$ grow to infinity can be viewed as an approximation to zero temperature limit. 

    The following results concern how the entropy of the ground states behave in our situation.
    \begin{proposition} \label{prop:zerotemperature}
        Let $(X,F)$ be a continuous flow on a compact metric space. Let $\ph\in C(X,\mathbb{R})$, $(\alpha_n)_{n\in \mathbb{N}}$ be a sequence growing to infinity, and $(\mu_{\alpha_n \ph})_{n\in \mathbb{N}}$ be a sequence of equilibrium states for $(\alpha_n \ph)_{n\in \mathbb{N}}$ respectively, that converges to some $\mu\in \mathcal{M}_F(X)$ in weak-* topology as $n\to \infty$. In addition, each $\mu_{\alpha_n \ph}$ is almost expansive at some given scale $\eps>0$. Then 
        \begin{enumerate}
            \item $\mu\in \mathcal{M}_{\max}(\ph)$;
            \item $h_{\mu}(F)=\sup\{h_{\nu}(F):\nu\in \mathcal{M}_{\max}(\ph)\}$;
            \item $\lim_{n\to \infty}h_{\mu_{\alpha_n\ph}}(F)=h_{\mu}(F)$.
        \end{enumerate}
    \end{proposition}
\begin{proof}
    For any $n\in \mathbb{N}$, since $\mu_{\alpha_n \ph}$ is an equilibrium state for $\alpha_n \ph$, for any $\nu\in \mathcal{M}_F(X)$, we have 
    $$
h_{\mu_{\alpha_n\ph}}(F)+\alpha_n\int \ph d\mu_{\alpha_n\ph}\geq h_{\nu}(F)+\alpha_n\int \ph d\nu.
    $$
    Dividing both sides by $\alpha_n$ and making $n\to \infty$, since $\alpha_n\to \infty$ and $\mu_{\alpha_n \ph}$ converges to $\mu$, we have $\int \ph d\mu\geq \int \ph d\nu$, which concludes (1) by the arbitrary choice of $\nu$.

    To prove (2), assume by contradiction that there exists some $\nu\in \mathcal{M}_{\max}(\ph)$ such that $h_{\nu}(F)>h_{\mu}(F)$. It then follows from Lemma \ref{lem:usc} that $h_{\nu}(F)>h_{\mu_{\alpha_n\ph}}(F)$ for all sufficiently large $n$. Since $\nu$ is $\ph$-maximizing, we also have $\int\ph d\nu\geq \int\ph d\mu_{\alpha_n \ph}$. Consequently, we have
    $$
h_{\nu}(F)+\alpha_n\int\ph d\nu>h_{\mu_{\alpha_n \ph}}(F)+\alpha_n\int\ph d\mu_{\alpha_n \ph}
    $$
   for all sufficiently large $n$, contradicting to the assumption that $\mu_{\alpha_n \ph}$ is an equilibrium state for $\alpha_n\ph$.

    Finally, to prove (3), notice that by Lemma \ref{lem:usc}, we only need to prove 
    \begin{equation} \label{eq:liminfentropy}
        \liminf_{n\to \infty} h_{\mu_{\alpha_n \ph}}(F)\geq h_{\mu}(F).
    \end{equation}
    Observe that $h_{\mu}(F)+\alpha_n\int \ph d\mu\leq h_{\mu_{\alpha_n \ph}}(F)+\alpha_n\int\ph d\mu_{\alpha_n \ph}$ for all $n$, which together with $\mu\in \mathcal{M}_{\max}(\ph)$ implies that $h_{\mu}(F)\leq h_{\mu_{\alpha_n \ph}}(F)$ for all $n$. This implies \eqref{eq:liminfentropy} immediately.
\end{proof}

\begin{lemma} \label{lem:limites}
    Let $F=(f_t)_{t\in \mathbb{R}}$ be a continuous flow over a compact metric space $X$, and let $\ph\in C(X,\mathbb{R})$ be a potential function with a unique equilibrium state $\mu_{\ph}$. Suppose $(\ph_n)_{n\in \mathbb{N}}$ is a sequence in $C(X,\mathbb{R})$, each with a unique equilibrium state $\mu_{\ph_n}$ such that
    \begin{itemize}
        \item there exists some $\eps>0$ such that each $\mu_{\ph_n}$ is almost expansive at scale $\eps$,
        \item $\lim_{n\to \infty}||\ph_n-\ph||_{C^0}=0$.
    \end{itemize}
    Then $(\mu_{\ph_n})_n$ converges to $\mu_{\ph}$ as $n\to \infty$ in weak-* topology.
\end{lemma}
\begin{proof}
    Let $\mu'$ be an arbitrary  weak-* limit point of $(\mu_{\ph_n})_n$. It suffices to show that $\mu'$ must be an equilibrium state for $\ph$. To see this, letting $(\mu_{n_k})_{k\in \mathbb{N}}$ be the subsequence of $(\mu_{\ph_n})_n$ converging to $\mu'$ as $k\to \infty$ in weak-* topology, we have from Lemma \ref{lem:usc} that 
    $$
h_{\mu'}(F)\geq \limsup_{k\to \infty}h_{\mu_{n_k}}(F).
    $$
    Meanwhile, we have
\begin{equation*}
\begin{aligned}
&\lim_{n\to \infty}\Big|\int \ph_{n_k}d\mu_{n_k}-\int \ph d\mu'\Big|\\
\leq &\lim_{n\to \infty}\Big(\Big|\int \ph_{n_k}d\mu_{n_k}-\int \ph d\mu_{n_k}\Big|+\Big|\int \ph d\mu_{n_k}-\int \ph d\mu'\Big|\Big)=0.    
\end{aligned}
\end{equation*}
    Consequently, we have
    $$
h_{\mu'}(F)+\int \ph d\mu'\geq \limsup_{k\to \infty}(h_{\mu_{n_k}}(F)+\int \ph_{n_k}d\mu_{n_k})=\limsup_{k\to \infty} P(\ph_{n_k})=P(\ph),
    $$
    where the last equality follows from the general fact $|P(\psi_1)-P(\psi_2)|\leq ||\psi_1-\psi_2||_{C^0}$. This gives that $\mu'=\mu_\ph$ by our assumption on $\ph$ having a unique equilibrium state. 
\end{proof}

\subsection{Functional analysis}

For a continuous flow $F$ on a compact metric space $X$, the maximum  functional $\Lambda_F$ defined in \eqref{max} is convex and continuous on $C(X,\RR)$. Below we recall an approximation theorem about tangent functionals to convex functionals. 

For a Banach space $V$ with norm $\|\cdot\|$,  its corresponding space $V^*$ of all real-valued bounded functionals on $V$ can be normed naturally by associating each element $\mu \in V^*$ with a number
  $$ 	
  \|\mu\|=\sup\{|\mu(f)|: f \in V, \|f\|=1\}.
  $$

Notice that for $\mu, \nu \in V^*$ we call:
\begin{itemize}
\item $\mu$ is \emph{convex} if $\mu(tf+(1-t)g) \le t\mu(f)+(1-t)\mu(g)$ for all $ f, g \in V$;
\item $\mu$ is \emph{bounded} by $\nu$ if $\mu(f) \le \nu(f), \forall f \in V$;
\item $\mu$ is \emph{tangent} to $\nu$ at $f \in V$ if $\mu(g)\le \nu(f+g)-\nu(f)$ for all $g\in V$.
\end{itemize}

\begin{theorem}(\cite[Theorem 3.1]{STY} \cite[Theorem V.1.1]{Isr}) \label{thm:tangentapproximation}
Let V be a Banach space  and let $ \Lambda \in V^*$  be continuous and convex. For any $\mu_0 \in V^*$ bounded by $\Lambda$, $f_0 \in V$ and $ \eps >0$, there exist $\mu \in V^*$ and $f\in V$ such that $\mu$ is tangent to $\Lambda$ at $f$ and 
   $$
   \|\mu-\mu_0\| \le \eps \,, \mbox{and}\,\,  \|f-f_0\| \le \frac{1}{\eps}(\Lambda(f_0)-\mu_0(f_0)+s),  
   $$
where $s=\sup\{\mu_0(g)-\Lambda(g): g\in V\} \le 0$.
\end{theorem}

The following lemma is well-known in the homeomorphism case; see \cite[Lemma 2.3]{Bre08}. The flow case shares the same idea. We include a proof in the Appendix for reader's convenience.

\begin{lemma} \label{lem:tangentmeasure}
    Let $F=(f_t)_{t\in \mathbb{R}}$ be a continuous flow over a compact metric space $X$, and let $\ph\in C(X,\mathbb{R})$. Then $\mu\in C(X,\mathbb{R})^*$ is tangent to $\Lambda_F$ at $\ph$ if and only if $\mu\in \mathcal{M}_F(X)$ and is $\ph$-maximizing. 
\end{lemma}

The ergodic decomposition theorem says that for every $\mu\in \mathcal{M}_F(X)$, $\mu=\int_{\mathcal{M}_F^e(X)} \nu d b_\mu(\nu)$, where $b_\mu$ is a Borel probability measure on $\mathcal{M}_F^e(X)$. The next lemma asserts that the barycenter map $\mu \mapsto b_\mu$ is isometric. 

\begin{lemma} \label{lem:barycentre map}(\cite[Corollary IV.4.2]{Isr})
Let $F$ be a continuous flow on compact metric space $X$. For any $\mu, \nu \in \mathcal{M}_F(X)$, we have 
	$$
	\|b_\mu-b_\nu\|=\|\mu-\nu\|.
	$$
\end{lemma}

Without difficulty, \cite[Lemma 3.3]{STY} can easily be generalized to the flow case. We simply cite the lemma below and omit the proof. 

\begin{lemma} \label{lem:barycentre support}(\cite[Lemma 3.3]{STY})
Let $F$ be a continuous flow on compact metric space $X$,  $\mu$ an invariant probability measure and $b_\mu$ just defined. We have 
\begin{itemize}
\item	If there exists a constant $C \ge 0$ such that $h(\nu) \le C$ for all $v\in supp(b_\mu)$, then $h(\mu) \le C$;
\item If $\mu \in \mathcal{M}_{\text{max}}(\ph)$, then $\text{supp}(b_\mu)$ is contained in  $\mathcal{M}_{\text{max}}(\ph)$.
\end{itemize}
\end{lemma}

\subsection{Some generic results for flows}

We also need the flow version of two generic results  \cite[Theorem 1.1]{Con} and \cite[Theorem 1.1]{Morris2}. The proofs are provided in the Appendix.

\begin{proposition} (Cf. \cite[Theorem 1.1]{Con})\label{thm:con}
Let $(X,F)$ be a continuous flow on a compact metric space, and $E$ be a topological vector space which is densely and continuously embedded in $C(X,\RR)$. 
Write
$$U(E):=\{\ph \in E: \text{there is a unique $\ph$-maximizing measure}\}.$$ 
Then $U(E)$ is a countable intersection of open and dense sets in $E$. If moreover $E$ is a Baire space, then $U(E)$ is dense in $E$.
\end{proposition}

\begin{proposition} (Cf. \cite[Theorem 1.1]{Morris2}) \label{thm:morris}
 Let $(X,F)$ be a continuous flow on a compact metric space. Suppose that $\mathcal{U}$ is an open and dense subset of $\overline{\mathcal{M}_F^e(X)}$. Then the set
$$U:= \{\ph\in C(X, \RR): \overline{\mathcal{M}_F^e(X)}\cap \mathcal{M}_{\max}(\ph)\subset \mathcal{U}\}$$ 
is open and dense in $C(X,\RR)$. 

Conversely, if $U\subset C(X, \RR)$ is open and dense, then the set 
$$\mathcal{U}:=\mathcal{M}_F^e(X)\cap \bigcup_{\ph\in U}\mathcal{M}_{\max}(\ph)$$ is open and dense in $\mathcal{M}_F^e(X)$.
\end{proposition}

\section{Proofs for Theorem \ref{thm:main}, \ref{thm:pathelogic}, and \ref{thm:groundstates}}\label{Proofs for Theorem}
This section is devoted to the proof of Theorems \ref{thm:main}, \ref{thm:pathelogic}, and \ref{thm:groundstates}. We first  introduce in $\mathsection 3.1$ some general initial setups that are shared by the proofs for both of Theorems \ref{thm:main} and \ref{thm:groundstates}, whose main strategy parallels which of \cite[Proposition 2.4]{STY}. Then we prove Theorems \ref{thm:main} and \ref{thm:groundstates} in $\mathsection 3.2$ and $\mathsection 3.3$ respectively. Finally in $\mathsection 3.4$, we apply the functional tools in $\mathsection 2.4$, as well as Theorem \ref{thm:groundstates}, to prove Theorem \ref{thm:pathelogic}.
\subsection{General setup}
Assume we are in the setup of Theorem \ref{thm:main}. Knowing $(X,F)$ satisfies non-uniform expansiveness of type \ref{E1} or \ref{E2}, we do not specify the exact type at this moment, and will make respective clarifications when there is a discrepancy regarding the arguments for them. 

For any constants $H_0^{\perp}\in (H^{\perp},h)$, $\delta_0>0$ and $n\in \NN$, let $\mu\in \mathcal{M}_F^e(X)$ be such that $h_\mu(F) > H_0^{\perp}$, and $U\subset \mathcal{M}_F(X)$ be given by
\begin{equation} \label{eq:defU}
    U:=\Big\{\nu\in \mathcal{M}_F(X):\Big|\int g_id\mu-\int g_id\nu\Big|<3\delta_0  \text{ for all } i\in \{1,\cdots,n\} \Big\}
\end{equation}
where $(g_i)_{i=1}^\infty$ form a basis of $C(X,\RR)$. Let $g_{\max}:=\max_{1\leq i\leq n}\|g_i\|_{C^0}$, and $\eta_g$ be a constant such that 
\begin{equation*}
    \sup\{|g_i(x)-g_i(y)|: x,y\in X,d(x,y)<\eta_g,1\leq i \leq n\}<\delta_0/8.
\end{equation*}
For $\mu \ae x\in X$, we have from Brin-Katok local entropy formula for flow without fixed point \cite{JCWZ19} and Birkhoff ergodic theorem that
$$
\lim_{\eps\to 0}\liminf_{t \to \infty}-\frac{1}{t}\log \mu(B_t(x,\eps))=\lim_{\eps\to 0}\limsup_{t \to \infty}-\frac{1}{t}\log \mu(B_t(x,\eps))=h_{\mu}(F)
$$
and 
$$
    \lim_{t\to \infty} \frac{1}{t}\int_0^tg_i(f_s(x))ds=\int g_id\mu,\ \  \forall 1 \le i \le n.
$$
Let $h_0:=h_{\mu}(F)-H_0^{\perp}$, and fix a positive constant
\begin{equation} \label{eq:eta03.1}
    \eta_0\in (0,\min\{h_0/24,\delta_0h_0/48g_{\max}\}).
\end{equation}
For $\mu \ae x$, let $\eps(x)>0$ be such that 
\begin{equation*}
\begin{aligned}
h_{\mu}(F)+\frac{\eta_0}{2}&>\limsup\limits_{t \to \infty}-\frac{1}{t}\log \mu(\overline{B_t(x,\eps)})\\
&>\liminf\limits_{t \to \infty}-\frac{1}{t}\log \mu(B_t(x,\eps))>h_{\mu}(F)-\frac{\eta_0}{2}
\end{aligned}
\end{equation*}
for all $0<\eps<\eps(x)$. Writing $X(\eps):=\{x\in X:\eps(x)>\eps\}$ for all $\eps>0$, let $\eps_0>0$ be small enough such that 
\begin{equation} \label{eq:conditioneps0}
    \mu(X(\eps_0))>1/2.
    \end{equation}
Then for every $x\in X(\eps_0)$, let $N(x)\in \mathbb{N}$ be the smallest integer such that
$$
-h_{\mu}(F)-\eta_0<\frac{1}{t}\log \mu(B_t(x,\eps_0/4))\leq \frac{1}{t}\log \mu(\overline{B_t(x,\eps_0)})<-h_{\mu}(F)+\eta_0
$$
and
\begin{equation} \label{eq:gidifferencedelta0}
    \max_{1\leq i \leq n}\left\{\Big|\frac{1}{t}\int_0^tg_i(f_s(x))ds-\int g_id\mu\Big|\right\}<\delta_0
\end{equation}
for all $t\geq N(x)$. For each $N\in \mathbb{N}$, let $X_N:=\{x\in X(\eps_0):N(x)\leq N\}$. It follows from \eqref{eq:conditioneps0} and Birkhoff ergodic theorem that there exists $N\in \mathbb{N}$ such that $\mu(X_N)>1/2$.
We will fix such $N$ for the rest of the paper. 

For any $\eps>0$, let $\Lambda_t(X_N,\eps)$ be the cardinality of maximally $(t,\eps)$-separated set for $X_N$. Let $t>N$. Given any maximally $(t,\eps_0/2)$-separated set $E_{t,\eps_0/2}$ for $X_N$, it is clear that elements in $\{B_t(x,\eps_0/4): x\in E_{t,\eps_0/2}\}$ are mutually disjoint, $\bigcup\limits_{x\in E_{t,\eps_0/2}}B_t(x,\eps_0/4)\subset X$, and $\mu(B_t(x,\eps_0/4))\geq e^{-t(h_{\mu}(F)+\eta_0)}$. Consequently, we have
\begin{equation} \label{eq:Etupperbound}
    \# E_{t,\eps_0/2}\leq e^{t(h_{\mu}(F)+\eta_0)}.
\end{equation}
Meanwhile, for any maximally $(t,\eps_0)$-separated set $E_{t,\eps_0}$ for $X_N$, we have $X_N\subset \bigcup\limits_{x\in E_{t,\eps_0}}\overline{B_t(x,\eps_0)}$ and $\mu(\overline{B_t(x,\eps_0)})\leq e^{-t(h_{\mu}(F)-\eta_0)}$, which imply that
\begin{equation} \label{eq:Etlowerbound}
     \# E_{t,\eps_0}\geq \mu(X_N)e^{t(h_{\mu}(F)-\eta_0)}.
\end{equation}
Let $X_{N,k}:=X_N\times \{k\}$ for each $k\in \mathbb{N}$. For any $a,b\in \mathbb{N}$ with $0\leq a+b \leq k$, let
$$
X_{N,k}(a,b):=\{x\in X_N:\lfloor p(x,k)\rfloor=a, \lfloor s(x,k)\rfloor=b\}
$$
and
$$
X_{N,k}(a,\cdot):=\bigcup_{b=0}^{k-a}X_{N,k}(a,b), \quad   X_{N,k}(\cdot,b):=\bigcup_{a=0}^{k-b}X_{N,k}(a,b).
$$
We write $\Lambda_{k}^{a,b}:=\Lambda_{k}(X_{N,k}(a,b),\eps_0)$. Similarly, we define $\Lambda_{k}^{a,\cdot}$ and $\Lambda_{k}^{\cdot,b}$ via $X_{N,k}(a,\cdot)$ and $X_{N,k}(\cdot,b)$ respectively. The next result indicates that orbit segments with large portion of prefix or suffix are non-essential concerning the entropy.
\begin{lemma} \label{lem:longtailentropy}
    There exist $C(h_0)>0$ and an integer $N_1\geq N$ such that for all $k\geq N_1$, we have
    $$
\max\Big\{\sum_{a\geq 6\eta_0k/h_0}\Lambda_{k}^{a,\cdot},\sum_{b\geq 6\eta_0k/h_0}\Lambda_{k}^{\cdot,b}\Big\}\leq C(h_0)e^{k(h_{\mu}(F)-3\eta_0)}.
    $$
\end{lemma}
\begin{proof}
We first deal with $\sum_{a\geq 6\eta_0k/h_0}\Lambda_{k}^{a,\cdot}$. For each integer $a\in [6\eta_0k/h_0,k]$, let $E_{k,\eps_0}(a,\cdot)$ be a maximally $(k,\eps_0)$-separated set for $X_{N,k}(a,\cdot)$, and
\begin{equation} \label{eqEaP}
    E_{a,\eps_0/2}^{\mathcal{P}}\subset X \text{ a maximally }(a,\eps_0/2) \text{-separated set for }X_N\cap [\mathcal{P}]_a.
\end{equation}
For each $x\in E_{k,\eps_0}(a,\cdot)$, there exists $x^{\mathcal{P}}\in E_{a,\eps_0/2}^{\mathcal{P}}$ such that $x\in \overline{B_a(x^{\mathcal{P}},\eps_0/2)}$. Given $x^{\mathcal{P}}\in E_{a,\eps_0/2}^{\mathcal{P}}$, we denote the collection of such $x\in E_{k,\eps_0}(a,\cdot)$ by $E_{k,\eps_0,a}(x^{\mathcal{P}})$. It is clear that $E_{k,\eps_0}(a,\cdot)=\bigcup\limits_{x^{\mathcal{P}}\in E_{a,\eps_0/2}^{\mathcal{P}}}E_{k,\eps_0,a}(x^{\mathcal{P}})$, and
\begin{equation}\label{contain}
\bigcup_{x\in E_{k,\eps_0,a}(x^{\mathcal{P}})}B_k(x,\eps_0/2)\subset B_a(x^{\mathcal{P}},\eps_0).
\end{equation}
Since $x^{\mathcal{P}}\in X_N$, when $k> Nh_0/6\eta_0$, we have $a>N$ and hence $\mu(B_a(x^{\mathcal{P}},\eps_0))\leq e^{-a(h_{\mu}(F)-\eta_0)}$. Meanwhile, by $E_{k,\eps_0}(a,\cdot)\subset X_N$, we also have $\mu(B_k(x,\eps_0/2))>e^{-k(h_{\mu}(F)+\eta_0)}$ for each $x\in E_{k,\eps_0}(a,\cdot)$. As elements in $\{B_k(x,\eps_0/2):x\in E_{k,\eps_0}(a,\cdot)\}$ are disjoint, by \eqref{contain} we have for each $x^{\mathcal{P}}\in E_{a,\eps_0/2}^{\mathcal{P}}$ that
\begin{equation}\label{less}
\#E_{k,\eps_0,a}(x^{\mathcal{P}}) \leq e^{-a(h_{\mu}(F)-\eta_0)+k(h_{\mu}(F)+\eta_0)}<e^{(k-a)h_{\mu}(F)+2k\eta_0}.
\end{equation}
Let $N_{\eta_0}>N$ be an integer such that $\Lambda_t([\mathcal{P}]\cup[\mathcal{S}],\eps_0/2)<e^{t(H^{\perp}+\eta_0)}$ for all $t>N_{\eta_0}$. Writing $N_1:=\lceil h_0N_{\eta}/6\eta_0 \rceil$, for $k>N_1$, it is also clear that for all $a\geq 6\eta_0k/h_0$, we have $\# E_{a,\eps_0/2}^{\mathcal{P}}\leq e^{a(H^{\perp}+\eta_0)}$. Combining with \eqref{less}, we have
\begin{equation*}
    \# E_{k,\eps_0}(a,\cdot)\leq \# E_{a,\eps_0/2}^{\mathcal{P}} (\max_{x^{\mathcal{P}}\in E_{a,\eps_0/2}^{\mathcal{P}}}\{\#E_{k,\eps_0,a}(x^{\mathcal{P}})\})<e^{-ah_0+kh_{\mu}(F)+3k\eta_0}.
\end{equation*}
By adding over $a\in [6\eta_0k/h_0,k]$ and applying $\eta_0<h_0/7$, we have
\begin{equation} \label{eq:headbeinga}
    \sum_{a\geq 6\eta_0k/h_0}\Lambda_{k}^{a,\cdot}<(1-e^{h_0})^{-1}e^{k(h_{\mu}(F)-3\eta_0)}.
\end{equation}

Now we turn to $\sum\limits_{b\geq 6\eta_0k/h_0}\Lambda_{k}^{\cdot,b}$. As above, for each integer $b\in [6\eta_0k/h_0,k]$, let $E_{k,\eps_0}(\cdot,b)$ be a maximally $(k,\eps_0)$-separated set for $X_{N,k}(\cdot,b)$, and 
\begin{equation} \label{eqEbS}
    E_{b,\eps_0/2}^{\mathcal{S}}\subset X \text{ a maximally }(b,\eps_0/2) \text{-separated set for }[\mathcal{S}]_b.
\end{equation}
Given any $b\in [6\eta_0k/h_0,k-N]$, for each $x\in E_{k,\eps_0}(\cdot,b)$, there exists $x^{\mathcal{S}}\in E_{b,\eps_0/2}^{\mathcal{S}}$ such that $f_{k-b}(x)\in \overline{B_b(x^{\mathcal{S}},\eps_0/2)}$. Again, for each $x^{\mathcal{S}}\in E_{b,\eps_0/2}^{\mathcal{S}}$, denote the collection of such $x$ by $E_{k,\eps_0}^b(x^{\mathcal{S}})$. It is clear that $E_{k,\eps_0}(\cdot,b)=\bigcup\limits_{x^{\mathcal{S}}\in E_{b,\eps_0/2}^{\mathcal{S}}}E_{k,\eps_0}^b(x^{\mathcal{S}})$, and $E_{k,\eps_0}^b(x^{\mathcal{S}})$ is $(k-b,\eps_0)$-separated for each $x^{\mathcal{S}}$. Consequently, we have from \eqref{eq:Etupperbound} and choice on $b$ that 
\begin{equation} \label{contain2}
\#E_{k,\eps_0}^b(x^{\mathcal{S}})\leq e^{(k-b)(h_{\mu}(F)+\eta_0)}, \quad \forall b\in [6\eta_0k/h_0,k-N].
\end{equation}
Meanwhile, for any $b>k-N$, as $k-b<N$, there exists a constant $C_{N}>0$ such that 
\begin{equation} \label{contain3}
\#E_{k,\eps_0}^b(x^{\mathcal{S}})< C_{N}, \quad \forall b\in [k-N,k].
\end{equation}
Since $\#E_{b,\eps_0/2}^{\mathcal{S}}\leq e^{b(H^{\perp}+\eta_0)}$ for all $b\geq 6\eta_0k/h_0$ with $k>N_1$, by \eqref{contain2} we have 
\begin{equation} \label{eq:bsmall}
    \#E_{k,\eps_0}(\cdot,b)\leq \# E_{b,\eps_0/2}^{\mathcal{S}} \max_{x^{\mathcal{S}}\in E_{b,\eps_0/2}^{\mathcal{S}}}\{\#E_{k,\eps_0}^b(x^{\mathcal{S}})\}\leq e^{-b(h_{\mu}(F)-H^{\perp})+kh_{\mu}(F)}
\end{equation}
for all $b\in [6\eta_0k/h_0,k-N]$, and by \eqref{contain3}
\begin{equation} \label{eq:bbig}
    \#E_{k,\eps_0}(\cdot,b)\leq \# E_{b,\eps_0/2}^{\mathcal{S}} \max_{x^{\mathcal{S}}\in E_{b,\eps_0/2}^{\mathcal{S}}}\{\#E_{k,\eps_0}^b(x^{\mathcal{S}})\} < C_{N}e^{kH^{\perp}+k\eta_0}
\end{equation}
for all $b\in (k-N,k]$. Adding \eqref{eq:bsmall} over $b\in [6\eta_0k/h_0,k-N]$ and \eqref{eq:bbig} over $b\in (k-N,k]$, it follows from $\eta_0<h_0/7$ that
\begin{equation} \label{eq:tailbeingb}
    \sum_{b\geq 6\eta_0k/h_0}\Lambda_{k}^{\cdot,b}<(1-e^{h_0})^{-1}e^{k(h_{\mu}(F)-6\eta_0)}+NC_{N}e^{k(h_{\mu}(F)-6\eta_0)}.
\end{equation}
The lemma is then concluded by combining \eqref{eq:headbeinga} and \eqref{eq:tailbeingb}.
\end{proof}
Write $E_{k,\eps_0}^{\leq 6\eta_0k/h_0}$ as a maximally $(k,\eps_0)$-separated set for the following set:
$$\{x\in X_N:\max\{\lfloor p(x,k)\rfloor,\lfloor s(x,k)\rfloor\}\leq 6\eta_0k/h_0\}.$$ 
As an immediate consequence of \eqref{eq:Etlowerbound} and Lemma \ref{lem:longtailentropy}, by increasing the size of $N_1$ if necessary, we have for each $k>N_1$ that
\begin{equation*}
    \#E_{k,\eps_0}^{\leq 6\eta_0k/h_0}>\frac{\mu(X_N)}{2}e^{k(h_{\mu}(F)-\eta_0)},
\end{equation*}
which implies the existence of integers $a_k,b_k\in [0,6\eta_0k/h_0]$ such that
\begin{equation} \label{eq:tailakbk}
    \Lambda_k(X_{N,k}(a_k,b_k),\eps_0)>\frac{\mu(X_N)}{2(6\eta_0k/h_0+1)^2}e^{k(h_{\mu}(F)-\eta_0)}.
\end{equation}

Let $\widetilde{X_{N,k}^{\mathcal{G}}(a,b)}:=\{(f_a(x),k-a-b):x\in X_{N,k}(a,b)\}$. Notice that $\widetilde{X_{N,k}^{\mathcal{G}}(a,b)}\subset \mathcal{G}^1$. We will sometimes abbreviate $\widetilde{X_{N,k}^{\mathcal{G}}(a_k,b_k)}$ as $\widetilde{X_{N,k}^{\mathcal{G}}}$ for all $k>N_1$. Notice that elements in $\widetilde{X_{N,k}^{\mathcal{G}}}$ all have length $c_k:=k-a_k-b_k$, which by the choices on $a_k,b_k$ and $\eta_0<h_0/24$ (see \eqref{eq:eta03.1}) satisfies $c_k>k/2$, and is therefore non-empty. Moreover, a combination of $H^{\perp}<h_{\mu}(F)$, our choice on $\eta_0$, and \eqref{eq:tailakbk} implies that
\begin{lemma} \label{lem:shorttailentropy}
    There exists an integer $N_2>N_1$ such that for all $k>N_2$, we have
    $$
\Lambda_{c_k}(\widetilde{X_{N,k}^{\mathcal{G}}(a_k,b_k)},\eps_0)>\frac{1}{2}e^{c_k(h_{\mu}(F)-2\eta_0)}.
    $$
\end{lemma}
\begin{proof}
Recall the definition of $E^{\mathcal{P}}_{a_k,\eps_0/2},E^{\mathcal{S}}_{b_k,\eps_0/2}$ from \eqref{eqEaP} and \eqref{eqEbS}. By \eqref{upper}, there exists $C_{\eta_0}=C_{\eta_0}([\mathcal{P}]\cup [\mathcal{S}],\eps_0/2)$ such that 
$$
\#E^{\mathcal{P}}_{a_k,\eps_0/2}<C_{\eta_0}e^{a_k(H^{\perp}+\eta_0}), \quad \#E^{\mathcal{S}}_{b_k,\eps_0/2}<C_{\eta_0}e^{b_k(H^{\perp}+\eta_0}), \quad \forall k\geq N_1.
$$
Notice that \eqref{eq:tailakbk} enables us to choose a $(k,\eps_0)$-separated set $E_{k,\eps_0}\subset X_{N,k}(a_k,b_k)$ such that $\#E_{k,\eps_0}>\frac{\mu(X_N)}{2(6\eta_0k/h_0+1)^2}e^{k(h_{\mu}(F)-\eta_0)}$. For each $x\in E_{k,\eps_0}$, there exists unique $(x',x'')\in E^{\mathcal{P}}_{a_k,\eps_0/2}\times E^{\mathcal{S}}_{b_k,\eps_0/2}$ such that 
$$
\max\{d_{a_k}(x',x),d_{b_k}(x'',f_{k-b_k}(x))\}<\eps_0/2.
$$
For each pair of $(x',x'')\in E^{\mathcal{P}}_{a_k,\eps_0/2}\times E^{\mathcal{S}}_{b_k,\eps_0/2}$, let 
$$E_{k,\eps_0}(x',x''):=\{x\in E_{k,\eps_0}:x'(x)=x',x''(x)=x''\}.$$ 
It follows immediately from definition that for each $(x',x'')$, $f^{a_k}E_{k,\eps_0}(x',x'')$ is $(c_k,\eps_0)$-separated. Therefore, by our choice on $\eta_0$ from \eqref{eq:eta03.1}, we have from pigeonhole principle and \eqref{eq:tailakbk} that
$$
\begin{aligned}
&\Lambda_{c_k}(\widetilde{X_{N,k}^{\mathcal{G}}(a_k,b_k)},\eps_0)\geq (\#E^{\mathcal{P}}_{a_k,\eps_0/2}\#E^{\mathcal{S}}_{a_k,\eps_0/2})^{-1}\#E_{k,\eps_0} \\
>&(2C_{\eta_0}^2(6\eta_0 k/h_0+1)^2)^{-1}\mu(X_N)e^{k(h_{\mu}(F)-\eta_0)-(a_k+b_k)(H^{\perp}+\eta_0)} \\
=&((2C_{\eta_0}^2(6\eta_0 k/h_0+1)^2)^{-1}e^{(a_k+b_k)(h_0-2\eta_0)})\mu(X_N)e^{c_k(h_{\mu}(F)-\eta_0)} \\
>&\mu(X_N)e^{c_k(h_{\mu}(F)-2\eta_0)}>\frac{1}{2}e^{c_k(h_{\mu}(F)-2\eta_0)}.
\end{aligned}
$$
where the second to the last inequality holds as long as $k$ is greater than some constant $N_2$, and the last one follows from our choice on $N$. This concludes the proof of the lemma.
\end{proof}

\begin{remark}
    So far everything holds as long as $h_{\mu}(F)>h([\mathcal{P}]\cup [\mathcal{S}])$. The necessity of having $h_{\mu}(F)>h^{\perp}_{\exp}$ for $(X,F)$ of type \ref{E2} will be revealed throughout $\mathsection 3.2$. 
\end{remark}

We end $\mathsection 3.1$ with an observation on the average for each $g_i$ along orbit segments from $\widetilde{X^{\mathcal{G}}_{N,k}}$. By our choice on $\eta_0$, for any $k>N_2$, $i\in \{1,...,n\}$ and $x\in\widetilde{X_{N,k}^{\mathcal{G}}(a_k,b_k)}$, we have
\begin{equation} \label{eq:kckintegraldifference}
    \begin{aligned}
    &\Big|\frac{1}{k}\int_0^k g_i(f_{s-a_k}(x))ds-\frac{1}{c_k}\int_0^{c_k}g_i(f_s(x))ds\Big| \\
    \leq &\Big|\frac{k-c_k}{k}\frac{\int_0^{c_k}g_i(f_s(x))ds}{c_k}\Big|+\Big|\frac{1}{k}\Big(\int_{-a_k}^0+\int_{c_k}^k\Big)g_i(f_s(x))ds\Big|\\
    \leq &\frac{2(a_k+b_k)g_{\max}}{k}\leq \frac{24\eta_0g_{\max}}{h_0}<\frac{\delta_0}{2},
\end{aligned}
\end{equation}
where the last inequality again follows from our choice on $\eta_0$ from \eqref{eq:eta03.1}. Then a combination of \eqref{eq:gidifferencedelta0} and \eqref{eq:kckintegraldifference} immediately gives that
\begin{equation} \label{eq:givariationforG}
    \max_{1\leq i \leq n}\Big\{\Big|\frac{1}{c_k}\int_0^{c_k}g_{i}(f_{s}(x))ds-\int g_id\mu\Big|\Big\}<3\delta_0/2.
\end{equation}

\subsection{Proof for Theorem \ref{thm:main}}
The main strategy for the proof of Theorem \ref{thm:main} parallels which of \cite[Theorem A(a)]{STY} in showing that for each $H\in [H^{\perp},h)$,  the following collection of measures 
$$
O_n:=\Big\{ \mu\in \overline{\mathcal{M}_F^e(X)}:h_{\mu}(F)<H+\frac{1}{n} \Big\}
$$
is open and dense for each $n$, which by Proposition \ref{thm:morris} leads to the desired result. For each $n$, such $O_n$ being open is well-known when $(X,F)$ is of type \ref{E1} as the entropy map $\mu \mapsto h_{\mu}(F)$ is upper semi-continuous in the entropy-expansive case, and a consequence of Corollary \ref{coro:Uhopen} when $(X,F)$ is of type \ref{E2}. Therefore, the bulk of the proof for Theorem \ref{thm:main} lies in showing denseness of such $O_n$. Indeed, such denseness result is a direct consequence of the following result, whose proof will occupy the rest of $\mathsection 3.2$. Recall \eqref{eq:defU} for the definition of $U$.
\begin{proposition} \label{prop:entropysmalldense}
For any $H_0^{\perp}\in (H^{\perp},h)$, there is some $\nu\in U$ such that $h_{\nu}(F)\leq H_0^{\perp}$.
\end{proposition}
\begin{proof}
Fix any $H_0^{\perp}$ as in the statement. In addition to our choice of $\eps_0$ from \eqref{eq:conditioneps0}, we assume that
\begin{equation} \label{eq:expansiveconditioneps0}
\begin{aligned}
    & (X,F) \text{ is entropy expansive at scale }\eps_0 \text{ when it is of type \ref{E1}, } \\
    & \text{and } h^{\perp}_{\exp}(\eps_0)<H_0^{\perp} \text{ when it is of type \ref{E2}.}
\end{aligned}
\end{equation}
Recall from Remark \ref{rmk:heta} that we denote by $h_{\eta}$ the gap function for $\mathcal{G}^1$ at scale $\eta$. Fix $\eta<\min\{\eps_0/8,\eta_g\}$, and some sufficiently large integer $k>N_2$ satisfying $c_k>\log k$, $h_{\eta}(k)/\log k<\delta_0/32g_{\max}$, and one additional condition that appears in the end of the proof (which will be satisfied as long as $k$ is large enough). In particular, we have 
\begin{equation} \label{eq:hetack}
    h_{\eta}(c_k)/c_k<h_{\eta}(k)/\log k<\delta_0/32g_{\max}.
\end{equation}
Let $(x_0,c_k)\in \widetilde{X^{\mathcal{G}}_{N,k}}$, whose existence is guaranteed by Lemma \ref{lem:shorttailentropy}.\footnote{We do not need the lower bound from Lemma \ref{lem:shorttailentropy} in $\mathsection 3.2$ as we only need one orbit segment with a sufficiently large $k$. The full power of the lemma will be revealed in $\mathsection 3.3$.} Since $(x_0,c_k)\in \mathcal{G}^1$, an application of Lemma \ref{lem:infinitespecification} to bi-infinite sequence of identical orbit segment $(x_0,c_k)$ at scale $\eta$ gives rise to $y_0\in X$ and a sequence of transition times $(h_i)_{i\in\mathbb{Z}}$ such that $d_{c_k}(x_0,f_{m_i}(y_0))\leq \eta$, where
\begin{equation} \label{eq:defmi}
    m_0=0, \quad m_i:=ic_k+\sum_{j=0}^{i-1}h_j, \quad m_{-i}:=-ic_k-\sum_{j=-i}^{-1}h_j, \quad \forall i\in \mathbb{N}.
\end{equation}
Let $Y_0:=\overline{\{f_t(y_0): {t\in \mathbb{R}}\}}$. 

We first show that every ergodic measure supported on $Y_0$ is in $U$. To see this, for any $y\in Y_0$ and $s>0$, by our definition of $Y_0$, there must be some $t=t(s,y)\in \mathbb{R}$ such that $d_s(y,f_t(y_0))<\eta<\eta_g$. Let $i_1=i_1(t)\in \mathbb{Z}$ be such that $t\in [m_{i_1-1},m_{i_1})$, and $i_2=i_2(t)\in \mathbb{Z}$ be such that $t+s\in [m_{i_2},m_{i_2+1})$. Writing $G_i^{c_k}(x):=\int_0^{c_k}g_i(f_u(x))du$ for all $x\in X$ and $i\in \{1,...,n\}$, as long as $i_2-i_1>16\delta_0/g_{\max}$, we have from \eqref{eq:givariationforG} and \eqref{eq:hetack} that
\begin{equation} \label{eq:measureonY0step1}
    \begin{aligned}
    &\Big|\int_t^{t+s}g_i(f_u(y_0))du-s\int g_id\mu\Big|\\
    =&\Big|\Big(\int_t^{m_{i_1}}+\sum_{l=i_1}^{i_2-1}(\int_{m_{l}}^{m_{l}+c_k}+\int_{m_{l}+c_k}^{m_{l}+c_k+h_{l}})+\int_{m_{i_2}}^{t+s}\Big)g_i(f_u(y_0))du-s\int g_id\mu\Big| \\
    \leq &\Big|\sum_{l=i_1}^{i_2-1}\int_{m_{l}}^{m_{l}+c_k}g_i(f_u(y_0))du-(i_2-i_1)c_k\int g_id\mu\Big|\\
    &+(2(i_2-i_1)h_{\eta}(c_k)+4c_k+4h_{\eta}(c_k))g_{\max} \\
    \leq &\Big|(i_2-i_1)G_i^{c_k}(x_0)-(i_2-i_1)c_k\int g_id\mu\Big|+(i_2-i_1)\text{Var}_{c_k}(g_i,\eta) \\
    &+2((i_2-i_1)h_{\eta}(c_k)+2c_k+2h_{\eta}(c_k))g_{\max} \\
    <&3(i_2-i_1)c_k\delta_0/2+(i_2-i_1)c_k\delta_0/8+4g_{\max}(i_2-i_1)c_k(h_{\eta}(c_k)/c_k+1/(i_2-i_1)) \\
    <&(i_2-i_1)c_k(3\delta_0/2+\delta_0/8+\delta_0/8+\delta_0/4)\leq 2\delta_0s.
\end{aligned}
\end{equation}
Notice that as long as $s>(3+16\delta_0/g_{\max})(h_{\eta}(c_k)+c_k)$, we must have $i_2-i_1>16\delta_0/g_{\max}$. Consequently, for all such $s$, we have
\begin{equation} \label{eq:measureinY0step2}
    \begin{aligned}
    &\Big|\int_0^{s}g_i(f_u(y))du-s\int g_id\mu\Big| \\
    \leq& \Big|\int_0^{s}g_i(f_u(y))du-\int_t^{t+s}g_i(f_u(y_0))du\Big|+\Big|\int_t^{t+s}g_i(f_u(y_0))du-s\int g_id\mu\Big| \\
    <&\delta_0s/8+2\delta_0s<3\delta_0s,
\end{aligned}
\end{equation}
which implies that every ergodic measure supported on $Y_0$ is in $U$ by the arbitrary choice on $y\in Y_0$ and Birkhoff Ergodic Theorem. This fact, by variational principle, reduces the proof of Proposition \ref{prop:entropysmalldense} to which of the following claim:

\textbf{Claim:} $h(Y_0)\leq H_0^{\perp}$.

We proceed the proof of the above claim by contradiction. Suppose that $h(Y_0)>H_0^{\perp}$. By our additional assumption \eqref{eq:expansiveconditioneps0} on $\eps_0$, it follows from \cite[Proposition 3.7]{CT16} and \cite[Proposition 2.7]{WangWu} respectively that
\begin{equation} \label{eq:entropyeps0Y0}
    h(Y_0,\eps_0/2)=h(Y_0)>H_0^{\perp}.
\end{equation}
Let $I:=\{m_i\}_{i\in \mathbb{Z}}$, where $m_i$ are as in \eqref{eq:defmi}. For each $t\in \mathbb{R}$, let $t_0=t_0(t)\in [0,c_k+h_{\eta}(c_k)]$ be the smallest non-negative number making $t+t_0\in I$, and then let $s_i:=\sum_{l=0}^ih_{n+l}$ for each $i\in \mathbb{N}$ provided $t+t_0=m_n$, and $s_{-1}:=0$. Consequently, each $t\in \mathbb{R}$ is associated with $\{t_0,\{s_i\}_{i\in \mathbb{N}}\}$. By compactness of $X$ and continuity of flow, there exists some constant $\beta\in (0,1)$ such that 
\begin{equation} \label{eq:defbeta} 
    \max_{x\in X,t\in [-\beta,\beta]}d(x,f_t(x))\ll \eps_0/8. 
\end{equation}
By shrinking $\beta$ if necessary, we may always assume that 
\begin{equation} \label{eq:defhetabeta}
    h_{\eta,\beta}:=h_{\eta}(c_k)/\beta\in \mathbb{N}
\end{equation}
Let
\begin{equation} \label{def:eqIm}
    I_m:=[m\beta,(m+1)\beta) \quad \text{ for each }m\in \mathbb{N}.
\end{equation}
Fix a $T>2(c_k+h_{\eta}(c_k))$. For any $t\in \mathbb{R}$, let $k_T=k_T(t)\in \mathbb{N}$ be such that
\begin{equation} \label{eq:defT}
    s_{k_T-1}+k_Tc_k\leq T-t_0 <s_{k_T}+(k_T+1)c_k.
\end{equation}
Letting $\mathbb{N}_{k,T}:=[\lfloor\frac{T}{c_k+h_{\eta}(c_k)}\rfloor-1,\lceil \frac{T}{c_k}\rceil]\cap \mathbb{N}$, it is clear that 
\begin{equation} \label{eq:NkTsize}
    k_T\in \mathbb{N}_{k,T}, \ \text{ and } \ \# \mathbb{N}_{k,T} \leq \frac{T}{c_k}-\frac{T}{c_k+h_{\eta}(c_k)}+4.
\end{equation}
For any $m\in \mathbb{N}_{k,T}$, let $K_m:=\{t\in \mathbb{R}:k_T(t)=m\}$. For each $t\in K_m$, let $I_m(t):=\{i_0,j_0,...,j_m\}\in \mathbb{N}^{m+2}$, where
\begin{itemize}
    \item $i_0$ is such that $t_0\in I_{i_0}$;
    \item $\{j_i\}_{i=0}^m$ is such that $t_0+s_i\in I_{j_i}$ for all $i\in \{0,1,...,m\}$. 
\end{itemize}
We also write $j_{-1}:=i_0$. Denote the collection of all possible $\{I_m(t):t\in K_m\}$ by $R_m$. Since $j_l-j_{l-1}\in [0,h_{\eta,\beta}]$, writing $c_{k,\beta}:=\lceil c_k/\beta \rceil$, we have
\begin{equation} \label{eq:Rmsize}
    \# R_m\leq (c_{k,\beta}+h_{\eta,\beta})(h_{\eta,\beta})^{m+1}.
\end{equation}
Given any such $\vec{r}_m:=\{i_0,j_0,...,j_m\}\in R_m$, let 
\begin{equation*}
    \begin{aligned}
T_m&=T_m(\vec{r}_m):=\{t\in K_m:I_m(t)=\vec{r}_m\},\\
Y_m&=Y_m(\vec{r}_m):=\{f_t(y_0):t\in T_m(\vec{r}_m)\}.
    \end{aligned}
\end{equation*}
We first evaluate $\Lambda_T(Y_m,\eps_0/2)$ for any given $\vec{r}_m$. Notice that $s_i-s_{i-1}\in [(j_i-j_{i-1}-1)\beta,(j_i-j_{i-1}+1)\beta]$ for all $i\in \{0,1,...,m\}$. Let
\begin{itemize}
    \item $J_0^{i_0}$ be a maximally $(i_0\beta,\eps_0/8)$-separated set of $X$,
    \item $J_i$ be $((j_i-j_{i-1}+1)\beta,\eps_0/8)$-separated of $X$ for all $i\in \{0,1,...,m\}$.
\end{itemize}
Writing $y_t:=f_t(y_0)$, we define a map $\pi:Y_m\to J_0^{i_0}\times \prod_{i=0}^m J_i$ by $\pi(y_t)=(z_{i_0},z_0,...,z_m)$, where 
\begin{itemize}
    \item $d_{i_0\beta}(z_{i_0},y_t)\leq\eps_0/8$,
    \item $d_{(j_{i}-j_{i-1}+1)\beta}(z_i,y_{t_0+t+(i+1)c_k+s_{i-1}})\leq\eps_0/8$ for all $i\in \{0,1,...,m\}$.
\end{itemize}
We want to investigate how $\pi$ behaves. Let $t_1,t_2\in T_m$ be such that $\pi(y_{t_1})=\pi(y_{t_2})$. We will evaluate $d(y_{t+t_1},y_{t+t_2})$ for each $t\in [0,T]$.
\begin{enumerate}
    \item If $t\leq i_0\beta$, we have 
    \begin{equation} \label{eq:dt2t1caseI}
        d(y_{t+t_1},y_{t+t_2})\leq d(y_{t+t_1},f_t(z_{i_0}))+d(y_{t+t_2},f_t(z_{i_0}))\leq \eps_0/4.
    \end{equation}
    \item If $t\in [t_0(t_1)+ic_k+s_{i-1}(t_1),t_0(t_1)+(i+1)c_k+s_{i-1}(t_1)]$ for some $i\in \{0,...,m-1\}$, writing $r_i:=t-(t_0(t_1)+ic_k+s_{i-1}(t_1))$, we have $r_i\in [0,c_k]$ and thus 
    \begin{equation} \label{eqyt1dist}
        d(y_{t_1+t},f_{r_i}(x_0))\leq \eta<\eps_0/8.
    \end{equation}
    Meanwhile, as $|t_0(t_2)+s_{i-1}(t_2)-(t_0(t_1)+s_{i-1}(t_1))|\leq \beta$, we know 
    $$
r_i+t_0(t_2)+s_{i-1}(t_2)+ic_k\in [t-\beta,t+\beta],
    $$
    which by \eqref{eq:defbeta} implies that
    \begin{equation} \label{eqyt2dist1}
        d(y_{t_2+t},y_{t_2+r_i+t_0(t_2)+s_{i-1}(t_2)+ic_k})<\eps_0/8.
    \end{equation}
    Additionally, we also notice from $r_i\in [0,c_k]$ that
    $$
d(y_{t_2+r_i+t_0(t_2)+s_{i-1}(t_2)+ic_k},f_{r_i}(x_0))\leq \eta<\eps_0/8,
    $$
    which together with \eqref{eqyt2dist1} implies that
    \begin{equation} \label{eqyt2dist2}
        d(y_{t_2+t},f_{r_i}(x_0))<\eps_0/4.
    \end{equation}
    Consequently, it follows from \eqref{eqyt1dist} and \eqref{eqyt2dist2} that
\begin{equation} \label{eq:dt2t1caseII}
    d(y_{t_2+t},y_{t_1+t}) <\eps_0/2.
\end{equation}
    \item If $t\in [t_0(t_1)+(i+1)c_k+s_{i-1}(t_1),t_0(t_1)+(i+1)c_k+s_{i}(t_1)]$ for some $i\in \{0,...,m-1\}$, writing $r_i':=t-(t_0(t_1)+(i+1)c_k+s_{i-1})$, we know $r_i'\in[0,(j_i-j_{i-1}+1)\beta]$ and thus
    $$
d(y_{t_1+t},f_{r_i'}(z_i))\leq \eta<\eps_0/8.
    $$
    Since $|t_0(t_2)+s_{i-1}(t_2)-(t_0(t_1)+s_{i-1}(t_1))|\leq \beta$, we know 
    $$
r_i'+t_0(t_2)+s_{i-1}(t_2)+(i+1)c_k\in [t-\beta,t+\beta].
    $$
    Together with $d(y_{t_2+r_i'+t_0(t_2)+s_{i-1}(t_2)+(i+1)c_k},f_{r_i'}(z_i))\leq \eta<\eps_0/8$, an argument parallel to which leading towards \eqref{eq:dt2t1caseII} gives that
    \begin{equation} \label{eq:dt2t1caseIII}
        d(y_{t_2+t},y_{t_1+t}) <\eps_0/2.
    \end{equation}
\end{enumerate}
A combination of \eqref{eq:dt2t1caseI}, \eqref{eq:dt2t1caseII} and \eqref{eq:dt2t1caseIII} implies that for $y_{t_1},y_{t_2}\in Y_m$ sharing the same image under $\pi$, we must have $d_T(y_{t_1},y_{t_2})<\eps_0/2$. In particular, this implies that for each $\vec{r}_m\in R_m$, we have
$$
\Lambda_T(Y_m,\eps_0/2)\leq \#J_0^{i_0}\prod_{i=0}^m \# J_i.
$$
We also know from \eqref{upper} that $\forall \zeta>0$, there exists some constant $C_{\zeta}=C_{\zeta}(X\times \mathbb{R}^+,\eps_0/8)$ such that 
\begin{equation} \label{eq:defCgamma}
    \Lambda_{s}(X,\eps_0/8)<C_{\zeta}e^{s(h+\zeta)} \quad \text{ for all }s\geq 0.
\end{equation}
Since $i_0\beta\leq c_k+h_{\eta}(c_k)$ and $(j_{i}-j_{i-1}+1)\beta\leq h_{\eta}(c_k)+\beta$ for all $i\in \{0,1,...,m\}$, by plugging in $\zeta=h$ into \eqref{eq:defCgamma}, it follows that $ \#J_0^{i_0}<C_he^{2h(c_k+h_{\eta}(c_k))}$ and $\# J_i<C_{h}e^{2h(h_{\eta}(c_k)+\beta)}$ for all $i$. Consequently, we have
\begin{equation} \label{eq:LambdaTYm'I}
    \Lambda_T(Y_m,\eps_0/2)< C_{h}^{m+2}e^{2h(c_k+(m+2)(h_{\eta}(c_k)+\beta))}.
\end{equation}
Writing $U_T:=\lceil\frac{T}{c_k}\rceil$, recall from \eqref{eq:NkTsize} that $m\leq U_T$. Note that $Y_0$ is the closure of the orbit of $y_0$, which is in turn partitioned into possible $Y_m$'s. It then follows from \eqref{eq:NkTsize}, \eqref{eq:Rmsize}, and \eqref{eq:LambdaTYm'I} that
\begin{equation*}
\begin{aligned}
    &\Lambda_T(Y_0,\eps_0/2) \leq \# \mathbb{N}_{k,T} \max\limits_{m\in \mathbb{N}_{k,T}}\# R_m \max\limits_{\vec{r}_m\in R_m}\{\# \Lambda_T(Y_m(\vec{r}_m),\eps_0/2)\}\\
    <&\Big(\frac{T}{c_k}-\frac{T}{c_k+h_{\eta}(c_k)}+4\Big)(c_{k,\beta}+h_{\eta,\beta})(h_{\eta,\beta})^{U_T+1}C_{h}^{U_T+2}e^{2h(c_k+(U_T+2)(h_{\eta}(c_k)+\beta))}.
\end{aligned}
\end{equation*}
By taking logarithm, dividing by $T$, and letting $T\to \infty$ on both sides of the above equation, we have
$$
h(Y_0,\eps_0/2)<\frac{2(\log (C_{h}h_{\eta}(c_k)/\beta)+h(h_{\eta}(c_k)+\beta)))}{c_k}<H_0^{\perp},
$$
as long as $k$ is sufficiently large. This reaches a contradiction to \eqref{eq:entropyeps0Y0}, and therefore the assumption on $h(Y_0)>H_0^{\perp}$. Consequently, we conclude the proof of the claim, as well as the proposition.
\end{proof}

\subsection{Proof of Theorem \ref{thm:groundstates}} Let $(X,F,\mu)$ be as in the statement of Theorem \ref{thm:groundstates}, $U$ be in the form of \eqref{eq:defU}, and $H_0\in (H^{\perp},h_{\mu}(F))$. Write $h_0':=h_{\mu}(F)-H_0$, and let
$$
\eta_0\in (0,\min\{h_0'/100,\delta_0h_0/48g_{\max}\}).
$$
Fix a constant $H_0^{\perp}\in (H^{\perp},H_0)$. As in \eqref{eq:expansiveconditioneps0}, in addition to \eqref{eq:conditioneps0}, we require $\eps_0$ to satisfy $h_{\exp}^{\perp}(\eps_0)<H_0^{\perp}$. In summary, writing $h_{1/2}:=(H_0+h_{\mu}(F))/2$, we have
\begin{equation} \label{eq:obstructionentropyeps0}
    h^{\perp}_{\exp}(\eps_0)<H_0^{\perp}<H_0<h_{1/2}-2\eta_0.
\end{equation}

Recall from Lemma \ref{lem:shorttailentropy} that for each sufficiently large $k$, we are able to find $c_k\in (k/2,k]$, $\widetilde{X_{N,k}^{\mathcal{G}}}\subset [\mathcal{G}]_{c_k}$, and a $(c_k,\eps_0)$-separated set $\widetilde{E_{c_k}}$ for $\widetilde{X_{N,k}^{\mathcal{G}}}$ such that 
\begin{equation} \label{eq:sizeEtilde}
    \# \widetilde{E_{c_k}}>\frac{1}{2}e^{c_k(h_{\mu}(F)-2\eta_0)}.
\end{equation}
Notice that $\widetilde{E_{c_k}}\times \{c_k\}\subset \mathcal{G}^1$. As in the proof of Proposition \ref{prop:entropysmalldense}, the key lies in choosing a subset of $\widetilde{E_{c_k}}$ with appropriate amount of elements, to which we apply specification and build a compact invariant subset $Y_1$ satisfying the following conditions:
\begin{itemize}
    \item $h(Y_1)<h$; in particular, $Y_1\subsetneqq X$. This is critical in defining $\varphi$, as well as establishing the pressure gap condition of $\alpha\ph$ for each $\alpha\geq 0$, which leads to the uniqueness of equilibrium states for all such potentials.
    \item $h(Y_1)>H_0$. This is important in describing the limit behavior of the above equilibrium states in terms of entropy.
\end{itemize}

Let us proceed to the construction of $Y_1$. Let $\eta\in (0,\min\{\eps_0/8,\eta_g\})$ and $k$ be large enough such that 
\begin{enumerate} 
    \item \label{1} $k>\frac{8\log2}{h_0'}$,
    \item \label{2} $\frac{h_{\eta}(k)}{k}<\frac{\eta_0}{4h_{1/2}}$ and $\frac{\log(h_{\eta}(k)/\beta+1)}{k}<\frac{\eta_0}{4}$, where $\beta$ is as in \eqref{eq:defbeta}.
    \item \label{3} $\log (2C_hh_{\eta}(k)/\beta)<\frac{k\eta_0}{4}$ and $h_{\eta}(k)+\beta<\frac{k\eta_0}{8h}$, where $C_{h}$ is the constant in \eqref{eq:defCgamma} with $\zeta=h$.
\end{enumerate}
In particular, it follows from conditions \eqref{1}, \eqref{2}, and \eqref{3} respectively that
\begin{equation} \label{eq:1consequence}
    \frac{1}{2}e^{c_k(h_{\mu}(F)-2\eta_0)}>e^{c_kh_{1/2}},
\end{equation}
\begin{equation*}
    \frac{c_k h_{1/2}}{c_k+h_{\eta}(c_k)}>h_{1/2}-\frac{\eta_0}{2} \quad \text{ and } \quad \frac{\log(h_{\eta,\beta}+1)}{c_k}<\frac{\eta_0}{2}
\end{equation*}
and
\begin{equation} \label{eq:3consequence}
   \frac{1}{c_k}\Big(\log(2C_hh_{\eta,\beta})+2h(h_{\eta}(c_k)+\beta)\Big)<\eta_0,
\end{equation}
where $h_{\eta,\beta}$ is as in \eqref{eq:defhetabeta}.

Let $E_{c_k}\subset \widetilde{E_{c_k}}$ be such that
$$
\# E_{c_k}\in [e^{c_kh_{1/2}},e^{c_k(h_{1/2}+\eta_0)}],
$$
which is plausible by \eqref{eq:sizeEtilde} and \eqref{eq:1consequence}. For any sequence $(x_i)_{i\in \mathbb{Z}}$ in $ \widetilde{E_{c_k}}$, since $\widetilde{E_{c_k}}\times \{c_k\}\subset \mathcal{G}^1$, we can apply Lemma \ref{lem:infinitespecification} at scale $\eta$ to obtain $y=y((x_i,c_k)_{i\in \mathbb{Z}})\in X$ and a sequence of transition times $(h_j)_{j\in \mathbb{Z}}:=(h_j((x_i)_{i\in \mathbb{Z}}))_{j\in \mathbb{Z}}$ such that $d_{c_k}(x_i,f_{m_i}(y))\leq \eta$, where the definition on 
$$(m_j)_{j\in \mathbb{Z}}:=(m_j((x_i)_{i\in \mathbb{Z}}))_{j\in \mathbb{Z}}$$
is as in \eqref{eq:defmi}. Let 
\begin{equation} \label{eq:Y1def}
    Y:=\Big\{y((x_i,c_k)_{i\in \mathbb{Z}}):(x_i)_{i\in \mathbb{Z}}\in \prod_{i\in \mathbb{Z}}{E_{c_k}}\Big\}, \quad \text{ and }\quad Y_1:=\overline{\{f_t(Y): t\in \mathbb{R}\}}.
\end{equation}
\begin{lemma} \label{lem:measureY1inU}
    Every ergodic measure supported on $Y_1$ is in $U$.
\end{lemma}
\begin{proof}
    The proof parallels which of the corresponding statement for $Y_0$ in Proposition \ref{prop:entropysmalldense}. For any $y_1\in Y_1$ and $s>0$, it follows from the definition of $Y_1$ that there must exist some $(x_i)_{i\in \mathbb{Z}}\in \prod_{i\in \mathbb{Z}}E_{c_k}$ and $t=t(y_1,s)\in \mathbb{R}$ such that $d_s(y_1,f_t(y))<\eta\leq \eta_g$, where $y=y((x_i,c_k)_i)$. Let $j_1=j_1(t)$, $j_2=j_2(t+s)$ be integers such that $t\in [m_{j_1-1},m_{j_1})$ and $t+s\in [m_{j_2},m_{j_2+1})$ respectively. As long as $s>(3+16\delta_0/g_{\max})(h_{\eta}(c_k)+c_k)$, we must have $j_2-j_1>16\delta_0/g_{\max}$. Then, a repetition of \eqref{eq:measureonY0step1} (the only difference is to replace $(j_2-j_1)G_i^{c_k}(x_0)$ by $\sum_{j=j_1}^{j_2-1}G_i^{c_k}(x_j)$) gives that
    $$
\Big|\int_t^{t+s}g_i(f_u(y))du-s\int g_id\mu\Big|\leq 2\delta_0s.
    $$
    As in \eqref{eq:measureinY0step2}, for all such $s$, we then have
    $$
\Big|\int_0^{s}g_i(f_u(y_1))du-s\int g_id\mu\Big|<3\delta_0s.
    $$
    This fact, together with our arbitrary choice on $y_1\in Y_1$, concludes the proof of the lemma by Birkhoff Ergodic Theorem.
\end{proof}
The following result provides us the desired entropy estimate for $Y_1$, as mentioned at the beginning of $\mathsection 3.3$.
\begin{proposition} \label{prop:entropyofY1}
    The topological entropy of $Y_1$ satisfies 
    $$h(Y_1)\in (h_{1/2}-\eta_0,h_{1/2}+2\eta_0).$$
\end{proposition}
\begin{proof}
    For any $(x_i)_{i\in \mathbb{Z}}\in \prod_{i\in \mathbb{Z}}E_{c_k}$, let 
    $$
I=I((x_i)_{i\in\mathbb{Z}}):=\{m_j((x_i)_{i\in \mathbb{Z}}): j\in \mathbb{Z}\}.
    $$
    Write $Z:=(\prod_{i\in \mathbb{Z}}E_{c_k})\times \mathbb{R}$. Given any $((x_i)_i,t)\in Z$, define $(t_0,(s_j)_{j\in \mathbb{N}})$ as in Proposition \ref{prop:entropysmalldense} by letting 
    \begin{itemize}
        \item $t_0=t_0((x_i)_i,t)\in [0,h_{\eta}(c_k)+c_k)$ be the smallest non-negative number to make $t+t_0\in I$; let $n\in \mathbb{Z}$ be such that $t+t_0=m_n$;
        \item $s_j=s_j((x_i)_i,t):=\sum_{l=0}^j h_{n+l}$.
    \end{itemize}
    Fix a constant $T>2(c_k+h_{\eta}(c_k))$. For each $((x_i)_i,t)\in Z$, let $k_T=k_T((x_i)_i,t)$ be as in \eqref{eq:defT}. By writing $L_T:=\lfloor\frac{T}{c_k+h_{\eta}(c_k)}\rfloor-1$ and recalling $U_T=\lceil \frac{T}{c_k}\rceil$, again we have
    $$
    k_T\in \mathbb{N}_{k,T}=[L_T,U_T].
    $$
    Meanwhile, for each $m\in \mathbb{N}$, let $K_m:=\{((x_i)_i,t):k_T((x_i)_i,t)=m\}$. Recall our choice on $\beta$ from \eqref{eq:defbeta}, as well as the definition for $(I_m)_{m\in \mathbb{N}}$ from \eqref{def:eqIm}. For each $m\in \mathbb{N}_{k,T}$ and $((x_i)_i,t)\in K_m$, define $\{i_0,j_0,...,j_m\}\in \mathbb{N}^{m+2}$ as before by letting 
    \begin{itemize}
    \item $i_0$ be such that $t_0\in I_{i_0}$;
    \item $\{j_i\}_{i=0}^m$ be such that $t_0+s_i\in I_{j_i}$ for all $i\in \{0,1,...,m\}$. 
\end{itemize}
Let $I_m((x_i)_i,t):=\{i_0,j_0,...,j_m\}$. For each $m\in \mathbb{N}_{k,T}$, denote the collection of all possible $\{I_m((x_i)_i,t):((x_i)_i,t)\in K_m\}$ by $R_{Z,m}$. As in \eqref{eq:Rmsize}, we have
\begin{equation} \label{eq:sizePm}
    \# R_{Z,m} \leq (c_{k,\beta}+h_{\eta,\beta})(h_{\eta,\beta})^{m+1}.
\end{equation}

We first turn to the proof of the lower bound of $h(Y_1)$. In this case, we concentrate on the set of $((x_i)_i,t)\in Z$ with $t=0$. Fix any $\hat{x}\in E_{c_k}$. Let 
\begin{equation*}
\begin{aligned}
E^T_{\hat{x}}&:=\{(x_i)_{i\in \mathbb{Z}}:x_i=\hat{x} \text{ for all }i\leq -1 \text{ and all\ } i\geq L_T-1\},\\
Y^T&:=\{y((x_i)_i):(x_i)_i\in E^T_{\hat{x}}\},\\
Y^T_m&:=\{y((x_i)_i)\in Y^T: ((x_i)_i,0)\in K_m\}.
\end{aligned}
\end{equation*}
Meanwhile, writing $R_{Z,m}^0:=\{\{i_0,j_0,...,j_m\}\in R_{Z,m}, i_0=0\}$, for any $m\in \mathbb{N}_{k,T}$ and $\vec{r}_m\in R_{Z,m}^0$, write $Y^T(\vec{r}_m):=\{y=y((x_i)_i)\in Y^T_m:I_m((x_i)_i,0)=\vec{r}_m\}$. We will show that provided fixed $m$ and $\vec{r}_m$, $Y^T(\vec{r}_m)$ is $(T,\eps_0/2)$-separated.

To see this, let $(x_i)_i,(x_i')_i\in E^T_{\hat{x}}$ be such that $y:=y((x_i)_i), y':=y((x'_i)_i)$ are both in $Y^T(\vec{r}_m)$. Let $n\in \mathbb{N}$ be the smallest integer where $x_n\neq x_n'$. Writing $m_n:=m_n((x_i)_i)$, $m_n':=m_n((x_i')_i)$, then $|m_n-m_n'|<\beta$ and we have
\begin{equation*}
\begin{aligned}
    &d_T(y,y')\geq d_{c_k}(f_{m_n}(y),f_{m_n}(y'))\\
    \geq &d_{c_k}(f_{m_n}(y),f_{m_n'}(y'))-d_{c_k}(f_{m_n'}(y'),f_{m_n}(y')) \\
    \geq &d_{c_k}(x_n,x_n')-d_{c_k}(f_{m_n}(y),x_n)-d_{c_k}(f_{m_n'}(y'),x_n')-d_{c_k}(f_{m_n'}(y'),f_{m_n}(y')) \\
    \geq &\eps_0-2\eta-\eps_0/8>\eps_0/2.
\end{aligned}
\end{equation*}
Consequently, we know $Y^T(\vec{r}_m)$ is $(T,\eps_0/2)$-separated. Meanwhile, noticing that $i_0=0$, there are at most $(\frac{T}{c_k}-\frac{T}{c_k+h_{\eta}(c_k)}+4)(h_{\eta,\beta}+1)^{U_T+1}$ choices on $(m,\vec{r}_m)\in \mathbb{N}_{k,T}\times R_{Z,m}^0$. Since $\#E^T_{\hat{x}}=(\#E_{c_k})^{L_T-1}\geq e^{c_k(L_T-1)h_{1/2}}$, there exist some $m$ and $\vec{r}_m$ such that
$$
\#Y^T(\vec{r}_m)\geq \Big((\frac{T}{c_k}-\frac{T}{c_k+h_{\eta}(c_k)}+4)(h_{\eta,\beta}+1)^{U_T+1}\Big)^{-1}e^{c_k(L_T-1)h_{1/2}},
$$
which implies that
$$
\Lambda_T(Y_1,\eps_0/2)\geq \Big((\frac{T}{c_k}-\frac{T}{c_k+h_{\eta}(c_k)}+4)(h_{\eta,\beta}+1)^{U_T+1}\Big)^{-1}e^{c_k(L_T-1)h_{1/2}}.
$$
Taking logarithm on both sides of the above equation, dividing by $T$ and making $T\to \infty$, we have 
$$
h(Y_1)\geq \frac{c_k h_{1/2}}{c_k+h_{\eta}(c_k)}-\frac{\log(h_{\eta,\beta}+1)}{c_k}>h_{1/2}-\eta_0.
$$
This provides us with the desired lower bound on $h(Y_1)$. 

Now let us turn to the upper bound. The idea parallels which for $h(Y_0)$ in the proof of Proposition \ref{prop:entropysmalldense}. Given any $(m,\vec{r}_m)\in \mathbb{N}_{k,T}\times R_{Z,m}$, denote by $Z(\vec{r}_m)$ the collection of $((x_i)_i,t)\in Z$ whose $\{i_0,j_0,...,j_m\}$ is equal to $\vec{r}_m$. Also recall the definitions for $J_0^{i_0}$ and $(J_i)_{i=0}^m$ as follows:
\begin{itemize}
    \item $J_0^{i_0}$ is a maximally $(i_0\beta,\eps_0/8)$-separated set of $X$,
    \item $J_i$ is $((j_i-j_{i-1}+1)\beta,\eps_0/8)$-separated of $X$ for all $i\in \{0,1,...,m\}$.
\end{itemize}
Let $y=y_t((x_i,c_k)_i):=f_t(y(x_i,c_k)_i)$. We define a map $\pi:Z(\vec{r}_m)\to J_0^{i_0}\times \prod_{i=0}^m(E_{c_k}\times J_i)$ by
$$
\pi((x_i),t):=(z_{i_0},(x_n,z_0),(x_{n+1},z_1),...,(x_{n+m},z_m))
$$
such that 
\begin{itemize}
    \item $d_{i_0\beta}(z_{i_0},y_t)<\eps_0/8$,
    \item $d_{(j_{i}-j_{i-1}+1)\beta}(z_i,y_{t_0+t+(i+1)c_k+s_{i-1}})<\eps_0/8$ for all $i\in \{0,1,...,m\}$.
\end{itemize}
A repetition of the process leading towards \eqref{eq:dt2t1caseI}, \eqref{eq:dt2t1caseII} and \eqref{eq:dt2t1caseIII} gives the following: given any $(m,\vec{r}_m)\in \mathbb{N}_{k,T}\times R_{Z,m}$, if $((x_i)_i,t),((x_i')_i,t')\in Z(\vec{r}_m)$ are such that $\pi((x_i)_i,t)=\pi((x_i')_i,t')$, we must have 
$$
d_T(f_t(y((x_i,c_k)_i)),f_{t'}(y((x_i',c_k)_i)))<\eps_0/2,
$$
which implies that 
$$
\Lambda_T(Z(\vec{r}_m),\eps_0/2)\leq \#J_0^{i_0}\prod_{i=0}^m(\#E_{c_k}\#J_i).
$$
Let $\zeta$ from \eqref{eq:defCgamma} be equal to $h$ once again. Following the deduction of \eqref{eq:LambdaTYm'I} and using $\#E_{c_k}\leq e^{c_k(h_{1/2}+\eta_0)}$, we have
\begin{equation} \label{eq:LambdaTZIm'}
     \Lambda_T(Z(\vec{r}_m),\eps_0/2)< C_h^{m+2}e^{2h(c_k+(m+2)(h_{\eta}(c_k)+\beta))}e^{(m+1)c_k(h_{1/2}+\eta_0)}.
\end{equation}
Since $Z=\bigcup\limits_{(m,\vec{r}_m)\in \mathbb{N}_{k,T}\times R_{Z,m}}Z(\vec{r}_m)$, it follows from \eqref{eq:NkTsize}, \eqref{eq:sizePm}, and \eqref{eq:LambdaTZIm'} that
\begin{equation*} 
\begin{aligned}
&\Lambda_T(Z,\eps_0/2)\leq \sum_{(m,\vec{r}_m)\in \mathbb{N}_{k,T}\times R_{Z,m}}\Lambda_T(Z(\vec{r}_m),\eps_0/2) \\
<&(\frac{T}{c_k}-\frac{T}{c_k+h_{\eta}(c_k)}+4)(h_{\eta,\beta}+c_{k,\beta})(h_{\eta,\beta}+1)^{U_T+1}C_h^{U_T+2} \\
&\cdot e^{2h(c_k+(U_T+2)(h_{\eta}(c_k)+\beta))}e^{(U_T+1)c_k(h_{1/2}+\eta_0)}.
\end{aligned}
\end{equation*}
Taking logarithm on both sides of the equation above, dividing by $T$ and making $T\to \infty$, we have 

\begin{equation*}
\begin{aligned}
h(Y_1,\eps_0/2)&\leq \frac{1}{c_k}\Big(\log(2C_{h}h_{\eta,\beta})+2h(h_{\eta}(c_k)+\beta)\Big)+h_{1/2}+\eta_0\\
&<h_{1/2}+2\eta_0, 
\end{aligned}
\end{equation*}
where the last inequality results from \eqref{eq:3consequence}. Meanwhile, we have from \eqref{eq:obstructionentropyeps0} that $h(Y_1)>h_{1/2}-\eta_0>h^{\perp}_{\exp}(\eps_0)$, which again by \cite[Proposition 3.7]{CT16} implies $h(Y_1,\eps_0/2)=h(Y_1)$. Consequently, this gives the desired upper bound for $h(Y_1)$.
\end{proof}

Now we proceed the proof of Theorem \ref{thm:groundstates} as follows.
\begin{itemize}
    \item Analyze the behavior of pressure function $P(\alpha):=P(\alpha \ph)$ for each $\alpha\geq 0$, where $\ph(x)=-d(x,Y_1)$ for each $x\in X$, i.e. the negative distance function to $Y_1$. Such analysis is conducted by studying thermodynamic formalism of $\{\alpha \ph\}_{\alpha\geq 0}$.
    \item Apply approximation theorems from functional analysis to make the above uncountably many equilibrium states become maximizing measures for some function.
\end{itemize}

Let $\varphi(x)=-d(x,Y_1)$ for any $x\in X$. Since $h(Y_1)<h$, we have $Y_1\subsetneq X$ and hence $\ph$ is not permanently $0$. Let us look into thermodynamic formalism of $(\alpha\ph)_{\alpha\geq 0}$ by checking whether conditions in Theorem \ref{thm:WW1} are satisfied.
\begin{itemize}
    \item Recall from \eqref{eq:obstructionentropyeps0} that $\eps_0>0$ is such that $h_{\exp}^{\perp}(\eps_0)<h_{1/2}-\eta_0<h(Y_1)$. Then we have for each $\alpha\geq 0$ that
    $$
P_{\exp}^{\perp}(\alpha\ph,\eps_0)\leq h_{\exp}^{\perp}(\eps_0)<h(Y_1)\leq P(\alpha \ph),
    $$
    where the last inequality follows from variational principle. Consequently, we have 
    \begin{equation} \label{eq:obstructionpressure}
        P_{\exp}^{\perp}(\alpha\ph)\leq P_{\exp}^{\perp}(\alpha\ph,\eps_0)<P(\alpha \ph) \quad \text{ for all }\alpha\geq 0.
    \end{equation}
    \item $P([\mathcal{P}]\cup [\mathcal{S}],\alpha \ph)\leq h([\mathcal{P}]\cup [\mathcal{S}])\leq H^{\perp}<h(Y_1)\leq P(\alpha \ph)$.
    \item It is straightforward that $\ph$ is Lipschitz. Therefore, by the additional assumption in Theorem \ref{thm:pathelogic}, $\ph$ satisfies Bowen property along $\mathcal{G}$. 
\end{itemize}

Consequently, the conditions in Theorem \ref{thm:WW1} are all satisfied for the family of potentials $(\alpha \ph)_{\alpha\geq 0}$. This indicates the existence and uniqueness of equilibrium state for each $\alpha \ph$, and thus concludes Theorem \ref{thm:groundstates}(1). Denote such measure by $\mu_{\alpha \ph}$.

We observe that Theorem \ref{thm:groundstates}(2) follows as an immediate consequence of \eqref{eq:obstructionpressure} and Lemma \ref{lem:limites}. To prove (3), it suffices to notice that $\mathcal{M}_{\max}(\ph)=\mathcal{M}_{\max}(\alpha \ph)=\mathcal{M}_F(Y_1)$ for all $\alpha>0$. Since $h(Y_1)>h_{1/2}-\eta_0>H_0$, a combination of Proposition \ref{prop:zerotemperature} and variational principle will conclude (3), following an immediate contradiction argument.

Finally, to prove Theorem \ref{thm:groundstates}(4), we will analyze convexity of the pressure function $P(\alpha):=P(\alpha \ph)$.  Since $\mu_{\alpha\ph}$ is the unique equilibrium state for $\alpha\ph$ for each $\alpha\ge0$, we have (cf. \cite{Wal})
\begin{equation}\label{deri}
P'(\alpha)=\int \ph d\mu_{\alpha \ph}.
\end{equation}
\begin{enumerate}
\item If $\alpha \mapsto P(\alpha)$ is strictly convex for all $\alpha\geq 0$, by \eqref{deri}, we know $\int \ph d\mu_{\alpha \ph}$ is strictly increasing as $\alpha\to \infty$, which gives the desired injectivity of $\alpha\mapsto \mu_{\alpha \ph}$ in this case. The rest of (4) follows immediately from (3) by choosing some appropriate $\alpha_1<\alpha_2$ with sufficiently large $\alpha_1$.

\item If $\alpha \mapsto P(\alpha)$ is as follows: for any $\hat{\alpha}>0$, there exist $\hat{\alpha}(1),\hat{\alpha}(2)$ satisfying $\hat{\alpha}<\hat{\alpha}(1)<\hat{\alpha}(2)<\infty$ such that $P$ is strictly convex for all $\alpha\in (\hat{\alpha}(1),\hat{\alpha}(2))$. Then we are able to build a sequence of intervals $((\hat{\alpha}_n(1),\hat{\alpha}_n(2)))_{n\in \mathbb{N}}$ such that for all $n\in \mathbb{N}$, we have $\hat{\alpha}_{n+1}(1)>\hat{\alpha}_n(2)$, and $P$ is strictly convex for all $\alpha\in (\hat{\alpha_n}(1),\hat{\alpha_n}(2))$. As in the first case, we know $\alpha\mapsto \mu_{\alpha \ph}$ is injective for all $\alpha\in \bigcup\limits_{n\in \mathbb{N}}(\hat{\alpha}_n(1),\hat{\alpha}_n(2))$. Then, by applying Lemma \ref{lem:limites} and arguing by contradiction once again, for all sufficiently large $n$, we have $\mu_{\alpha \ph}\in U$ for all $\alpha\in (\hat{\alpha}_n(1),\hat{\alpha}_n(2))$. (4) is then concluded by letting $\alpha_1:=\hat{\alpha}_n(1)$, $\alpha_2:=\hat{\alpha}_n(2)$ with $n$ being large enough.
    
\item If $\alpha \mapsto P(\alpha)$ is in neither of the above two cases, due to $C^1$ smoothness of $P(\alpha)$ there must exist some smallest $\alpha_0\geq 0$ such that $P(\alpha)$ is linear for all $\alpha\geq \alpha_0$. In this case, we see from \eqref{deri} and the variational principle that that $\mu_{\alpha_0\ph}$ is the unique equilibrium state for all $\alpha \ph$ with $\alpha\geq \alpha_0$. In particular, it follows from Proposition \ref{prop:zerotemperature} that $\mu_{\alpha_0\ph}$ is a measure of maximal entropy for $(Y_1,F)$. If $\alpha_0=0$, then $\mu_{\alpha_0\ph}$ would be a measure of maximal entropy for $(X,F)$, a contradiction to $h(Y_1)<h$. So we must have $\alpha_0>0$. Consequently, it follows from our choice of $\alpha_0$ that there must exist some $\delta\in (0,\alpha_0)$ such that the pressure $P(\alpha)$ is strictly convex for all $\alpha\in (\alpha_0-\delta,\alpha_0)$, which implies injectivity of $\alpha\mapsto \mu_{\alpha \ph}$ for all such $\alpha$. Meanwhile, a combination of Proposition \ref{prop:zerotemperature} and Lemma \ref{lem:limites} implies that
    $$
\lim_{\alpha\uparrow \alpha_0}\mu_{\alpha \ph}=\mu_{\alpha_0 \ph} \quad \text{ and } 
    $$
    (4) is then concluded by making $\alpha_1$ sufficiently close to $\alpha_0$ from below, and $\alpha_2:=\alpha_0$.
\end{enumerate}
A summation of all three possible cases from above enables us to conclude (4), and thus the proof of Theorem \ref{thm:groundstates}.

\subsection{Proof of Theorem \ref{thm:pathelogic}}

We will apply the functional tool introduced in $\mathsection 2.4 $, as well as Theorem \ref{thm:groundstates}, to prove Theorem \ref{thm:pathelogic}. 

For any $\ph \in C(X,\RR)\setminus \RRR_H$, there exists $\tilde\mu\in \mathcal{M}_{\max}(\ph)$ with $h_{\tilde\mu}(F)>H$. Moreover, due to Lemma \ref{lem:barycentre support}, $\tilde\mu$ can be assumed ergodic. For $\eps>0$, consider the following open neighborhood of $\tilde\mu$ 
$$
\mathcal{U}:=\Big\{\nu \in \mathcal{M}_{F}(X): \int \ph d\nu > \int \ph d\tilde\mu -\eps^2\Big\}.
$$
Applying Theorem \ref{thm:groundstates}, there exist a Lipschitz continuous function $\psi_0 \in C(X, \RR)$, a family of equilibrium states $\mu_{\alpha \psi_0}$ for the potential $\alpha \psi_0$ and a pair of constants $0< \alpha_1 <\alpha_2 <\infty$ such that the map $\alpha \mapsto \mu_{\alpha \psi_0}$ is injective and continuous from $[\alpha_1, \alpha_2]$ to $\mathcal{U}$ and $h_{\mu_{\alpha \psi_0}}(F)>H$ for all such $\alpha$. 

Using the subscript set $[\alpha_1, \alpha_2]$, we may construct a measure $\hat{\mu}$ on $\mathcal{M}_F^{e}(X)$ by assigning $$
 \hat{m}(\cdot)=m(\{t \in [\alpha_1, \alpha_2]: \mu_{\alpha \psi_0}\in \cdot\}) 
 $$
 where $m$ denotes the normalized Lebesgue measure on $[\alpha_1, \alpha_2]$. This measure is well defined and non-atomic due to injectivity and continuity of the map $\alpha \mapsto \mu_{\alpha \psi_0}$. Put $\mu_0 =\int_{\mathcal{M}_F^e(X)}\nu d\hat{m}(\nu)$. Ergodic decomposition theorem claims that for every invariant measure $\lambda \in \mathcal{M}_F(X)$ there is a unique probability measure $b_\lambda$
 on $\mathcal{M}_F^e(X)$ such that $\lambda =\int_{\mathcal{M}_F^e(X)}\nu d b_\lambda(\nu)$. By the uniqueness, $b_{\mu_0}=\hat{m}$ and $\mu_0$ belongs to $\mathcal{U}$.  As an element of $C(X, \RR)^*$, $\mu_0$ is bounded from above by the convex maximum functional $\Lambda_F$ (see \eqref{max}). By Theorem \ref{thm:tangentapproximation}, there exist $\psi \in C(X, \RR)$ and $\mu \in C(X, \RR)^*$ such that $\mu$ is tangent to $\Lambda_F$ at $\psi$ and 
    $$
   \|\mu-\mu_0\| \le \eps \quad \mbox{and} \quad  \|\psi-\ph\|_{C^0} \le \frac{1}{\eps}(\Lambda_F(\ph)-\mu_0(\ph)) \le \eps.
   $$
 According to Lemma \ref{lem:tangentmeasure}, $\mu $ maximizes $\psi$. It suffices to show the support of $b_\mu$, $\text{supp}(b_\mu)$, contains uncountably many ergodic measures. 
 
Notice that $\mathcal{M}_F^e(X)$ as a subspace of $\mathcal{M}_F(X)$ is metrizable and $b_{\mu_{0}}$ is outer regular. For any $\eps>0$, there exists an open subset $U$ of $\mathcal{M}_F^e(X)$ such that $\supp(b_{\mu}) \subset  U$ and $b_{\mu_0}(U \setminus \supp(b_\mu)) < \eps$.  As $\nu\mapsto b_\nu$ is isometric, $\|b_\mu-b_{\mu_0}\|=\|\mu-\mu_0\| \le \eps$. We have
 $$
 b_{\mu_0}(\text{supp}(b_{\mu})) \ge b_{\mu_0}(U) -\eps \ge b_\mu(U)-2\eps =1-2\eps.
 $$
 Since $b_{\mu_0}$ is non-atomic, it follows that $\supp(b_\mu)$ contains uncountable many elements. 

\section{Examples and applications}\label{Examples and applications}

\subsection{Geodesic flows}
Let $M$ be a smooth closed $n$-dimensional manifold, equipped with a smooth Riemannian metric $g$. Denote by $\bar \pi: SM\to M$ the unit tangent bundle over $M$. For each $v\in SM$,
there exists a unique geodesic, denoted by $\gamma_{v}: \RR\to M$, satisfying the initial condition $\dot \gamma_v(0)=v$. The geodesic flow $(g^{t})_{t\in\mathbb{R}}$ on $SM$ is then defined as:
\[
g^{t}: SM \rightarrow SM, \quad v \mapsto \dot \gamma_{v}(t),\ \ \ \ \forall\ t\in \RR .
\]

A \emph{Jacobi field} along a geodesic $\gamma:\RR\to M$ is a vector field $J(t)$ which satisfies the following \emph{Jacobi equation}
\[J''+R(J, \dot \gamma)\dot \gamma=0,\]
where $R$ is the Riemannian curvature tensor and\ $'$\ means the covariant derivative along $\gamma$.
A Jacobi field $J(t)$ along a geodesic $\gamma$ is called \emph{parallel} if $J'(t)=0$ for all $t\in \RR$. 
\begin{definition}
For each $v \in SM$, the rank of $v$, denoted by \text{rank}($v$), is defined to be the dimension of the vector space of parallel Jacobi fields along the geodesic $\gamma_{v}$. Then the rank of $M$ is defined as 
\[\text{rank}(M):=\min\{\text{rank}(v): v \in SM\}.\] 
For a geodesic $\gamma$ we define $\text{rank}(\gamma):=\text{rank}(\dot \gamma(t))$ for some $t\in \mathbb{R}$ (and hence for all $t\in \mathbb{R}$).
\end{definition}

We always assume that $(M,g)$ is a closed rank one Riemannian manifold of nonpositive (sectional) curvature everywhere. Then $SM$ splits into two subsets invariant under the geodesic flow: the regular set $\text{Reg}:= \{v\in SM: \text{rank}(v)=1\}$, and the singular set $\text{Sing}:= SM \setminus \text{Reg}$. Every ergodic measure supported on $\text{Reg}$ is a hyperbolic measure and hence exhibits nonuniform hyperbolicity (cf. \cite[Corollary 3.7]{BCFT}). See \cite[Chapter 12]{BP} and \cite[Section 2.4]{BCFT} for more details on the ergodic theory of geodesic flows in nonpositive curvature.

In particular, we have two $g^t$-invariant subbundles $E^s$ and $E^u$ of $TSM$, which are integrable into $g^t$-invariant foliations $W_g^s$ and $W_g^u$ respectively. For $v\in SM$, we call $W_g^{s/u}(v)$ the \emph{stable/unstable manifolds} of the geodesic flow through $v$. The projections $H^{s/u}(v)=\bar\pi W_g^{s/u}(v)$ are called the \emph{stable/unstable horospheres} associated to $v$. Denote by $\UUU_v^{s/u}: T_{\bar\pi v}H^{s/u}\to  T_{\bar\pi v}H^{s/u}$ the symmetric linear operator associated to the stable/unstable horosphere $H^{s/u}$. Let $\lambda^u(v)$ be the minimal eigenvalue of $\UUU_v^u$ and let $\lambda^s(v):=-\lambda^u(-v)$. Then we define $\lambda(v):=\min \{\lambda^u(v), \lambda^s(v)\}.$ It is clear that $\lambda: SM\to \RR$ is a continuous function. 

Let $\eta>0$. Following \cite{BCFT}, we define 
\begin{equation*}
    \begin{aligned}
\GGG(\eta)&:=\{(v,t): \int_0^\tau \lambda(g^sv)ds \ge \eta\tau, \int_0^\tau \lambda(g^{-s}g^tv)ds \ge \eta\tau, \forall \tau \in [0,t]\},\\
\BBB(\eta)&:=\{(v,t): \int_0^t \lambda(g^sv)ds < \eta t\}.
    \end{aligned}
\end{equation*}
\begin{definition}(\cite[p. 1221]{BCFT})\label{decom}
Given $(v,t)\in SM\times\RR^+$, take $p=p(v,t)$ to be the largest time such that $(v,p)\in \BBB(\eta)$, and $s=s(v,t)$ to be the largest time in $[0, t-p]$ such that $(g^{t-s}v, s)\in \BBB(\eta)$. Then it follows that $(g^pv, l)\in \GGG(\eta)$ where $l=t-p-s$. Thus the triple $(\BBB(\eta), \GGG(\eta), \BBB(\eta))$ equipped with the functions $(p,l,s)$ determines a decomposition for $SM\times \RR^+$.
\end{definition}

Using the above $(\PPP,\GGG,\SSS)$-decomposition for the geodesic flow and the Climenhaga-Thompson criterion Theorem \ref{thm:CT}, Burns, Climenhaga, Fisher and Thompson \cite{BCFT} proved the uniqueness of equilibrium states for H\"older continuous potentials satisfying certain gap properties.
\begin{theorem}(Cf. \cite[Theorem A]{BCFT})\label{bcftthm}
Let $g^t$ be the geodesic flow over a closed rank one manifold $M$ and let $\varphi :SM\to \RR$ be H\"{o}lder continuous.
If $P(\text{Sing},\varphi) <P(\varphi)$, then $\varphi$ has a unique equilibrium state $\mu$. This equilibrium state is hyperbolic, fully supported, and is the weak$^{*}$ limit of weighted regular closed geodesics.
\end{theorem}

We are ready to prove the genericity results for maximal measures of the geodesic flow.
\begin{proof}[Proof of Theorem \ref{thm:bcft}]
Let us verify the conditions of Theorem \ref{thm:main} for geodeisc flows on rank one manifolds of nonpositive curvature.
\begin{itemize}
    \item By \cite[Theorem 6.1]{Kn2} and \cite[Theorem B]{BCFT}, we have $h_{\text{top}}(\text{Sing})<h_{\text{top}}(G)=:h$. 
    \item By \cite[Proposition 5.4]{BCFT}, $h^\perp_{\exp}\le h_{\text{top}}(\text{Sing})$ and hence $h^\perp_{\exp}<h$. 
    \item By \cite[Proposition 5.2]{BCFT}, for $\eta>0$ small enough, $h([\PPP]\cup[\SSS])=h([\BBB(\eta)])<h$.
\end{itemize} 
Thus Theorem \ref{thm:main} applies and Theorem \ref{thm:bcft}(1) is proved.

To prove Theorem \ref{thm:main}(2)-(5), we need to show that every Lipschitz function $\psi$ satisfies Bowen property along $\GGG$. But it is just an immediate consequence of \cite[Corollary 7.5]{BCFT}. Note that our potential $\psi:=\alpha \ph$ is a $\alpha$-multiple ($\alpha>0$) of distance function $\ph$ defined as follows
$$\ph(v):=-d(v,Y_1), \quad \forall v\in SM,$$
where $Y_1$ is a compact flow-invariant subset of $SM$ with entropy $h(Y_1)\in (h(\text{Sing}), h)$ (see \eqref{eq:Y1def} and Proposition \ref{prop:entropyofY1}). Thus Theorems \ref{thm:pathelogic} and  \ref{thm:groundstates} apply. It remains to show that the ergodic measure in \ref{thm:bcft}(2) are all fully supported. This follows from Theorem \ref{bcftthm}, that is, the unique equilibrium state for $\psi$ above is fully supported. The proof of Theorem \ref{thm:bcft} is complete.
\end{proof}

\subsection{Frame flows}
Let $(M,g)$ be a closed, oriented, $C^{\infty}$ and $n$-dimensional Riemannian manifold. The \emph{frame bundle} over $M$ is defined as
\begin{equation*}
\begin{aligned}
FM:=\{&(x, v_0, v_1, \cdots, v_{n-1}): x\in M, v_i\in S_xM, \\
&\{v_0, \cdots, v_{n-1}\} \text{\ is a positively oriented orthonormal frame at\ } x\}.
\end{aligned}
\end{equation*}
The \emph{frame flow} $F^t:FM\to FM$ is defined by
$$F^t(x, v_0, v_1, \cdots, v_{n-1}):= (\gamma_{v_0}(t), g^t(v_0), P_\gamma^t(v_1),\cdots, P_\gamma^t(v_{n-1}))$$
where $P_\gamma^t$ is the parallel transport along the geodesic $\gamma_{v_0}$.
There is a natural fiber bundle $\pi: FM\to SM$ that takes a frame to its first vector, and then each fiber can be identified with $SO(n-1)$. Clearly, $\pi \circ F^t=g^t\circ \pi$ and $F^t$ acts isometrically along the fibers.

The ergodicity of frame flows on closed manifolds of negative curvature has been extensively studied since 1970s by Brin, Gromov \cite{Brin75, Brin80, BG} etc. In this case, the frame flow $F^t$ is a partially hyperbolic flow, and the uniqueness of equilibrium states has been studied by Spatzier and Visscher, see \cite[Theorem 1]{SV}.

We consider frame flows on closed rank one manifolds of nonpositive curvature, which can be thought as systems with mixture of partiall hyperbolicity and nonuniform hyperbolicity. We say $M$ is \emph{bunched} if there exists a constant $C>1$ such that for every $x\in M$ and $\{v_i\}_{i=1}^4\subset S_xM$, we have 
    $$
K(v_1,v_2)\in [C^{-1}K(v_3,v_4),CK(v_3,v_4)]
    $$
    whenever the sectional curvature $K$ is defined.

\begin{definition}\label{decom1}
Given $(x,t)\in FM\times \RR^+$, then $(\pi(x), t)\in  SM\times \RR^+$ has a $(\PPP, \GGG, \SSS)$-decomposition according to Definition \ref{decom}. We define  $(\tilde \PPP, \tilde\GGG, \tilde\SSS)$-decomposition of $(x,t)$ as the unique lift of  $(\PPP, \GGG, \SSS)$-decomposition of $(\pi(x), t)$.
\end{definition}

Using the above $(\tilde \PPP, \tilde\GGG, \tilde\SSS)$-decomposition for frame flows and Theorem \ref{thm:weprove}, we have
 \begin{theorem}(\cite[Theorem E]{WangWu})\label{esframeflow}
Let $M$ be a closed, oriented rank one $n$-manifold with bunched nonpositive curvature. Suppose that the frame flow $F^t:FM\to FM$ is topologically transitive, and $\tilde\varphi: FM\to \mathbb{R}$ is a continuous potential that is constant on the fibers of the bundle $FM\to SM$, whose projection $\varphi :SM\to \RR$ is H\"{o}lder continuous. Moreover, assume that $P(g^t, \text{Sing},\varphi) <P(g^t, \varphi)$. Then there is a unique equilibrium measure for $(F^t, \tilde\varphi)$. 
\end{theorem}

The following is a version of Bowen's inequality \cite[Theorem 17]{Bow71} for pressure of orbit segments of the frame flow. It is essentially due to the fact that the flow is isometric along fibers, and hence there is no contribution to pressure in the fiber direction.
\begin{lemma}(\cite[Lemma 4.28]{WangWu})\label{liftpressure}
For $\tilde{\mathcal{C}}\subset FM\times \RR^+$, denote $\mathcal{C}:=\{(\pi x, t): (x,t)\in \tilde{\mathcal{C}}\}\subset SM\times \RR^+$. Then for any continuous potential $\tilde\varphi: FM\to \mathbb{R}$ that is constant on the fibers of the bundle $\pi: FM\to SM$, we have 
$$P(F^t, \tilde{\mathcal{C}},\tilde\varphi)=P(g^t,\mathcal{C},\varphi).$$
\end{lemma}

At last, we complete the proof of Theorem \ref{thm:frame}.
\begin{proof}[Proof of Theorem \ref{thm:frame}]
To prove Theorem \ref{thm:frame}(1), let us verify the conditions of Theorem \ref{thm:main} for frame flows. Note that the frame flow is entropy expansive (cf. \cite[Proof of Theorem F]{WangWu}). By Lemma \ref{liftpressure}, we have
$$H^{\perp}:=h([\tilde\PPP]\cup[\tilde\SSS])=h([\PPP]\cup[\SSS])<h_{\text{top}}(G)=h_{\text{top}}(F)=:h.$$
Thus Theorem \ref{thm:main} applies and Theorem \ref{thm:frame}(1) is proved.

For Theorem \ref{thm:frame}(2)-(5), we cannot apply Theorems \ref{thm:pathelogic} and \ref{thm:groundstates} directly due to the failure of $h^{\perp}_{\exp}<h$. In fact, one can easily observe that $\Gamma_{\eps}(x)\nsubseteq f_{[-s,s]}(x)$ for any $\eps,s>0$ and $x\in X$, causing $\text{NE}(\eps)=X$ for all $\eps>0$. Consequently, $h^{\perp}_{\exp}=h$. Nevertheless, we can mimic their proofs to get Theorem \ref{thm:frame}(2)-(5). 
Let $Y_1\subset SM$ be the set built in the case of $(SM,(g^t)_{t\in \mathbb{R}})$ as in \eqref{eq:Y1def}, which has appeared in the proof of Theorem \ref{thm:bcft}. 
As before, let $\varphi:SM\mapsto \mathbb{R}$ be given by 
$$\varphi(v):=-d(v,Y_1), \quad \forall v\in SM.$$
Then we define the function $\tilde \ph: FM\to \RR$ as
$$\widetilde{\varphi}(x):=\varphi(\pi(x)), \quad \forall x\in FM.$$
Now we have the following properties:
\begin{itemize}
    \item $\widetilde{\mathcal{G}}$ satisfies weak specification, by \cite[Theorem 4.27]{WangWu}.
    \item $\widetilde{\varphi}$ satisfies Bowen property over $\widetilde{\mathcal{G}}$, since $\varphi$ satisfies such property over $\mathcal{G}$, and $\widetilde{\varphi}$ is a constant on each fiber.
    \item For each $\alpha>0$, we have $P([\tilde\PPP]\cup[\tilde\SSS],\alpha\widetilde{\ph})\leq h([\tilde\PPP]\cup[\tilde\SSS])=h([\PPP]\cup[\SSS])<h(Y_1)=h(\widetilde{Y_1})\leq P(\alpha\widetilde{\ph})$, where $\widetilde{Y_1}:=\pi^{-1}(Y_1)$.
\end{itemize}
Then by Theorem \ref{esframeflow}, there exists a unique equilibrium state $\mu_{\alpha\widetilde{\ph}}$ for each $\alpha\ge 0$. This allows us to mimic the proof of Theorem \ref{thm:groundstates} and obtain all desired results aside from measures from Theorem \ref{thm:frame} (2) being fully supported. 

It remains to prove that these measures are fully supported. Since such measures are all in the form of $\mu_{\alpha\widetilde{\ph}}$ with $\alpha>0$, it suffices to show all such equilibrium states are fully supported. The bulk of the proof lies in showing such measures all satisfy lower Gibbs property along long orbit segments from $\widetilde{\mathcal{G}}$ at sufficiently small scale, i.e. there is some constant $\eps_0>0$ such that for all $\alpha>0$ , there is a constant $T_{\alpha}>0$ such that the following holds for all $(x,t)\in \widetilde{\mathcal{G}}$ with $t>T_{\alpha}$: for any $\eps'\in (0,\eps_0)$, there is $C_{\alpha,\eps'}>0$ such that
\begin{equation} \label{eq:lowergibbs}
    \mu_{\alpha\widetilde{\ph}}(B_t(x,\eps'))>C_{\alpha,\eps'}e^{-tP(\alpha\widetilde{\ph})+\int_0^t\alpha\widetilde{\ph}(F^s(x))ds}.
\end{equation}
To prove \eqref{eq:lowergibbs}, we first notice from \cite[(4.26)]{WangWu} that 
\begin{equation} \label{eq:Ftscale}
    P(F^t,\alpha\widetilde{\ph},\eps')=P(F^t,\alpha\widetilde{\ph}) \text{ for all }\alpha>0, \\ \eps'\in (0,\eps_0)
\end{equation}
where $\eps_0:=\text{inj}(M)/6$. Then we can directly apply \eqref{eq:Ftscale} as well as the above three conditions (on weak specification, Bowen property and pressure gap condition respectively) over $(F^t,\alpha\widetilde{\ph},\tilde\PPP,\tilde\GGG,\tilde\SSS)$ to run the arguments leading towards \cite[Lemma 4.16]{CT16}. Precisely, we obtain for each $\alpha\geq 0$ an equilibrium state $\nu_{\alpha}$ for $\alpha\widetilde{\ph}$ satisfying the inequality suggested in \eqref{eq:lowergibbs}. Since we have known the uniqueness of equilibrium state, such $\nu_{\alpha}$ must be equal to $\mu_{\alpha\widetilde{\ph}}$.\footnote{In their original work of \cite{CT16}, in order to prove Lemma 4.16, the only expansiveness type of condition Climenhaga and Thompson need is Proposition 3.7, which helps them fix a scale on which all the arguments are based on. Our \eqref{eq:Ftscale} serves as an analog of that proposition and plays an identical role. Indeed, the full power on the condition on obstruction of expansiveness will not be revealed till \cite[$\mathsection 4.6$]{CT16}, where they start to implement the proof for uniqueness of equilibrium state.} Once we have \eqref{eq:lowergibbs}, we can repeat the arguments in \cite[Lemma 6.2 and Proposition 6.3]{BCFT} and conclude that all $\mu_{\alpha\widetilde{\ph}}$ must be fully supported for all $\alpha>0$. Details are omitted and left to interested readers.
\end{proof}

\section{Appendix: Proofs of some technical results for flows}

We prove Lemma \ref{lem:tangentmeasure}, Proposition \ref{thm:con} and Proposition \ref{thm:morris}, by adapting the proofs of \cite[Lemma 2.3]{Bre08}, \cite[Theorem 1.1]{Con} and \cite[Theorem 1.1]{Morris2} from the homeomorphism case to the flow case.

\begin{proof}[Proof of Lemma \ref{lem:tangentmeasure}]
The ``if'' direction is clear by the fact $\mu\in \mathcal{M}_F(X)$ and observing that $\Lambda_F(\ph)=\mu(\ph)$ and
$$\mu(\psi)=\mu(\ph+\psi)-\mu(\ph)\le \Lambda_F(\ph+\psi)-\Lambda_F(\ph).$$ 
So we only need prove the ``only if'' direction. By definition of tangency to $\Lambda_F$, for any $\psi\in C(X,\mathbb{R})$, we have
\begin{equation} \label{eq:deftangency}
    \mu(\psi)\leq \Lambda_F(\ph+\psi)-\Lambda_F(\ph).
\end{equation}
Replacing $\psi$ by $-\psi$ in \eqref{eq:deftangency}, we have
\begin{equation} \label{eq:deftangency-}
    \mu(\psi)\geq \Lambda_F(\ph)-\Lambda_F(\ph-\psi).
\end{equation}
Plugging in $\psi:=1$ in both \eqref{eq:deftangency} and \eqref{eq:deftangency-} immediately gives that $\mu(1)=1$. Meanwhile, plugging $\psi=\ph$ in those two equations indicates that
$$
\mu(\ph)=\Lambda_F(\ph).
$$
Therefore, it remains to show that $\mu$ is flow invariant. For any $t\neq 0$, and $\psi_1\in C(X,\mathbb{R})$, letting $\psi:=\psi_1-\psi_1\circ f_t$, we know $\Lambda_T(\ph+\psi)=\Lambda_T(\ph-\psi)=\Lambda_T(\ph)$. Consequently, a combination of \eqref{eq:deftangency} and \eqref{eq:deftangency-} in this case implies that
$$
\mu(\psi_1-\psi_1\circ f_t)=\mu(\psi)=0,
$$
which concludes that $\mu\in \mathcal{M}_F(X)$ by the arbitrary choice on $t,\psi_1$.
\end{proof}

\begin{proof}[Proof of Proposition \ref{thm:con}]
For compact metric space $X$, $C(X, \RR)$ is separable.  Since $E$ is densely and continuously embedded in $C(X, \RR)$, we may assume there is a countable subset of $E$ which is dense in $C(X, \RR)$ and denote it as $\{g_i\}_{i=1}^{\infty}$. An invariant measure $\mu$ is uniquely determined by how it integrates the family $\{g_i\}_{i=1}^{\infty}$ by Riesz representation theorem. Therefore the set of $\ph$-maximizing measure $\mathcal{M}_{\max}(\ph)$ is a singleton if and only if the closed interval $M_i(\ph)=\{\int g_i d\mu:  \mu \in \mathcal{M}_{\max}(\ph)\}$ is a singleton for every $i\ge 1$. Define 
$$E_{ij}=\{\ph \in E: \|M_i(\ph)\|\ge j^{-1}\}$$
where$ \|M_i(\ph)\| $ denotes the length of the interval $M_i(\ph)$. Then we have $U(E)^c=\cup_{i=1}^{\infty}\cup_{j=1}^{\infty} E_{ij}$ and hence $U(E)=\cap_{i=1}^{\infty}\cap_{j=1}^{\infty} E_{ij}^c$ . It suffices to show $E_{i, j}$ is closed and has empty interior. 

To show $E_{ij}$ is closed, let $\ph_\alpha$ be a convergent net in $E_{ij}$ and $\ph$ be its limit. We can write $M_i(\ph_\alpha)=[\int g_i d\mu_\alpha^-,  \int g_i d\mu_\alpha^+]$ for some $ \mu_\alpha^-,  \mu_\alpha^+ \in \mathcal{M}_{\max}(\ph_\alpha)$. By compactness of $\mathcal{M}_F(X)$, without loss of generality we assume $\mu_\alpha^- \rightarrow \mu^-, \mu_\alpha^+ \rightarrow \mu^+$ in weak* topology.  Since
$$
\int \ph_\alpha d\mu_\alpha^- =\int (\ph_\alpha-\ph) d\mu_\alpha^- +\int \ph d\mu_\alpha^- \rightarrow \int \ph d\mu^{-},
$$ 
and $ \int \ph_\alpha d\mu \rightarrow \int \ph d\mu$ for any $\mu \in \mathcal{M}_F(X)$, $ \int \ph d\mu^{-}\ge \int \ph d\mu $ due to $\mu_\alpha^- \in \mathcal{M}_{\max}(\ph_\alpha)$ and $\mu^-$ is $\ph$-maximizing follows. Similarly $\mu^+$ is $\ph$-maximizing. Moreover, $\int g_i d\mu_\alpha^+ - \int g_i d\mu_\alpha^- \rightarrow \int g_i d\mu^+- \int g_i d\mu^- \ge j^{-1}$ gives $\ph \in E_{ij}$. 

To show $E_{ij}$ has empty interior, we follow the same perturbation argument in \cite{Jen06a}. For any $\ph \in E_{ij}$ and  $\eps > 0$, the set 
$$M_i(\ph +\eps g_i)=\Big\{\int g_i d\mu: \mu \in \mathcal{M}_{\max}(\ph+\eps g_i)\Big\} $$ 
is a closed interval $ [\int g_i d\mu_\eps^-, \int g_i d\mu_\eps^+]$ for some $\mu_\eps^-$ and $ \mu_\eps^+$. Again we may assume $\{\mu_\eps^{\pm}\}_{\eps>0}$ is a convergent net with limit $\mu^{\pm}$,  otherwise we can just pick any convergent sub-net. Since
$$\int \ph+\eps g_i  d\mu_\eps^-  \ge \int \ph+\eps g_i  d\mu$$
for arbitrary $\mu \in \mathcal{M}_F(X)$,  $\int \ph d\mu^- \ge \int \ph d\mu$ as $\eps \rightarrow 0$ and $\mu^-$ is $\ph$-maximizing follows.  For any $\nu \in \mathcal{M}_{\max}(\ph)$, $\int \ph+\eps g_i  d\mu_\eps^-  \ge \int \ph+\eps g_i  d\nu$ implies $\int  g_i  d\mu_\eps^- \ge  \int g_i  d\nu$ and $\int g_i d\mu^- \ge \int g_i d\nu$ by taking the limit.  Therefore $\int g_i d\mu^- =\max_{\mu \in \mathcal{M}_{\max}(\ph)}\int g_id\mu$.  Similarly we can show $\int g_i d\mu^+ =\max_{\mu \in \mathcal{M}_{\max}(\ph)}\int g_id\mu$. 
Hence the set $M_i(\ph +\eps g_i)$  converges to a singleton 
$\{\int g_i d\mu^+\}$ in Hausdorff metric and $\ph+\eps g_i \notin E_{ij}$ for sufficiently small $\eps$.
\end{proof}

We need two lemmas to prove Proposition \ref{thm:morris}. The following lemma says a measure almost maximizing $\ph$ can be realized as a maximizing measure for some $g$ close to $\ph$. The proof of the next lemmas is adaption of that in homeomorphism case, and thus omitted. 
	\begin{lemma}(Cf. \cite[Lemma 2.2]{Morris2})\label{close}
	Let $\ph \in C(X, \RR)$ and $ \eps>0$.  If $\mu \in \mathcal{M}_F^{e}(X)$ satisfies $\Lambda_F(\ph)-\int \ph d\mu<\eps$, then there exists $g\in C(X,\RR)$ such that $\mu \in \mathcal{M}_{\max}(g)$ and $\|\ph-g\|_{C^0}< \eps$.
	\end{lemma}
Moreover, any ergodic measure is uniquely maximizing for some continuous function due to O. Jenkinson \cite{Jen06a}.
	\begin{lemma}(Cf. \cite[Theorem 1]{Jen06a})\label{thm:jen}
	Let $\mu \in \mathcal{M}_{F}^{e}(X)$. Then there exists $\ph \in C(X,\RR)$ such that $\mathcal{M}_{\max}(\ph)=\{\mu\}$.
	\end{lemma}	

\begin{proof}[Proof of Proposition \ref{thm:morris}]
Given a convergent sequence $ \ph_n \rightarrow \ph $ in $C(X,\RR)$, let $\mu_n \in \mathcal{M}_{\max}(\ph_n)$ be a sequence of corresponding maximizing measures. Without loss of generality, we assume that $\{\mu_n\}$ converges in weak-* topology to $\mu$. The following inequality holds for any invariant measure $\nu$ and implies that $\mu$ maximizes $\ph$:
	$$\int\ph d\nu-\|\ph-\ph_n\|_{C^0} \le \int \ph_n d\nu \le \int \ph_n d\mu_n \le \int \ph d\mu_n +\|\ph-\ph_n\|_{C^0}.
	$$
Obviously $U^c=\{\ph\in C(X, \RR): \overline{\mathcal{M}_F^e(X)}\cap \mathcal{M}_{\max}(\ph)\cap \mathcal{U}^c\neq \emptyset \}$. Take any convergent sequences $\ph_n \in U^c \rightarrow \ph$ and $\mu_n \in \mathcal{M}_{\max}(\ph_n)\cap \mathcal{U}^c \rightarrow \mu$. As $\mathcal{U}^c$ is closed, we have $\mu \in \mathcal{U}^c$ and $\ph \in U^c$. So $U$ is open.

Next we show $$\mathcal{U}:=\mathcal{M}_F^e(X)\cap \bigcup_{\ph\in U}\mathcal{M}_{\max}(\ph)$$ is open when $U$ is open. Take any $\mu \in \mathcal{U}$. Let $ \ph \in U$ such that $\mu \in \mathcal{M}_{\max}(\ph)$ and let $\mathcal{V} \subset \mathcal{M}_{F}^{e}(X)$ be an open neighborhood of $\mu$ small enough such that $\Lambda_F(\ph)-\int \ph d\nu < \eps$ for any $\nu \in \mathcal{V}$. According to Lemma \ref{close}, there exists $g$ maximized by $\nu$ with $\|\ph-g\|_{C^0} < \eps$ for any $\nu \in \mathcal{V}$. Choosing smaller $\eps$ if needed, we have $g\in U$ and $\nu \in \mathcal{U}$  by openness of $U$. So $\mathcal{V} \subset \mathcal{U}$ and hence $\mathcal{U}$ is open. 

To show $U$ is dense when $\mathcal{U}$ is dense, it suffices to prove $U$ intersects every nonempty open subset $V\subset C(X,\RR)$. We already prove that $\mathcal{V}:=\overline{\mathcal{M}_F^e(X)}\cap \bigcup_{\ph\in V}\mathcal{M}_{\max}(\ph)$ is nonempty open subset in $\overline{\mathcal{M}_{F}^{e}(X)}$. There exists $\mu \in \mathcal{V}\cap \mathcal{U}$ by density of $\mathcal{U}$. Also there exists $g\in V$ maximized by $\mu$. Meanwhile  Lemma \ref{thm:jen} says we can find $h\in C(X,\RR)$ uniquely maximized by $\mu$. Due to homogeneity, $\mathcal{M}_{\max}(\eps h)=\mathcal{M}_{\max}(h)=\{\mu\}$ for any $\eps >0$.  This implies $\mathcal{M}_{\max}(g+\eps h)=\{\mu\}$ and for small enough $\eps$. Therefore, $g+\eps h \in U \cap V$. This proves the density of $U$.

Conversely, let $\mathcal{V}$ be an open subset in $\mathcal{M}_{F}^{e}(X)$ and $\mu$ be an ergodic element in $\mathcal{V}$. By Lemma \ref{thm:jen}  there exists $ \ph \in C(X,\RR)$ uniquely maximized by $\mu$. We already show that $V=\{\ph\in C(X, \RR): \overline{\mathcal{M}_F^e(X)}\cap \mathcal{M}_{\max}(\ph)\subset \mathcal{V}\}$ is open in $C(X,\RR)$. Obviously $\ph \in V$, therefore $V$ is nonempty and open. Since $U$ is dense, it follows that $U \cap V \neq \emptyset $ and consequently $\mathcal{U} \cap \mathcal{V} \neq \emptyset $. Therefore, $\mathcal{U}$ is dense. 
\end{proof}
    
\ \
\\[-2mm]
\textbf{Acknowledgement.}
This work is supported by National Key R\&D Program of China No. 2021YFA1003204, NSFC Nos. 12401239, 12201536, 12071474, and STCSM No. 24ZR1437200.

\end{document}